\documentclass[12pt]{amsart}

\usepackage{amsmath,amssymb,amsfonts}
\usepackage{amsthm}
\usepackage{mathrsfs}
\usepackage{tikz}
\usepackage{tikz-cd}
\usepackage[margin=1cm]{geometry}

\usetikzlibrary{arrows.meta,positioning}

\usetikzlibrary{calc,3d,shapes,arrows,decorations.markings}
\tikzset{midarrow/.style = {postaction=decorate, decoration={markings,mark=at position .5 with \arrow{stealth}}}}

\newcommand{\Obj}{\operatorname{Obj}}
\newcommand{\Ee}{\mathcal{E}}

\title[Nuclear dimension and pure infiniteness for higher rank graph $C^*$-algebras]
{Nuclear dimension, pure infiniteness and real rank for higher rank graph $C^*$-algebras}

\author{David Pask}

\address{College of Science and Engineering, James Cook University,
Townsville, QLD 4811, Australia}
\email{david.a.pask@gmail.com}



\date{\today}

\theoremstyle{plain}
\newtheorem{theorem}{Theorem}[section]
\newtheorem{lemma}[theorem]{Lemma}
\newtheorem{proposition}[theorem]{Proposition}
\newtheorem{corollary}[theorem]{Corollary}

\theoremstyle{definition}
\newtheorem{definition}[theorem]{Definition}
\newtheorem{example}[theorem]{Example}
\newtheorem{examples}[theorem]{Examples}

\theoremstyle{remark}
\newtheorem{remark}[theorem]{Remark}

\newcommand{\N}{\mathbb{N}}

\newcommand{\Prim}{\mathrm{Prim}}

\begin{document}

\begin{abstract}
We study the structure and regularity of higher rank graph $C^*$-algebras, with
particular emphasis on their nuclear dimension. For a row-finite, locally convex
$k$-graph $\Lambda$ with no sources, we characterise pure infiniteness of
$C^*(\Lambda)$ in terms of generalised cycles, maximal tails, and strong
aperiodicity, and we relate these conditions to topological dimension zero of the
primitive ideal space and to the structure of gauge-invariant ideals. Our main
application is that whenever $C^*(\Lambda)$ is purely infinite of topological
dimension zero---in particular whenever its ideal lattice is finite---it is
strongly purely infinite, $\mathcal O_\infty$-stable, and of nuclear dimension
one, \emph{even when $C^*(\Lambda)$ is not simple}. This extends to the
non-simple, higher-rank setting the nuclear-dimension-one computation known for
simple UCT-Kirchberg $2$-graph algebras. Along the way we refine and correct
several results in the existing graph $C^*$-algebra literature.
\end{abstract}

\maketitle

\section{Introduction}\label{sec:intro}

Higher-rank graph $C^*$-algebras were introduced by Kumjian and Pask
\cite{KP2} as a common generalisation of the Cuntz--Krieger algebras, the
graph $C^*$-algebras of \cite{KPR,KPRR,BPRS}, and the rank-$k$ Cuntz--Krieger
algebras arising from buildings. To a $k$-graph $\Lambda$---a countable
category carrying a degree functor $d\colon\Lambda\to\N^k$ with a unique
factorisation property---one associates the universal $C^*$-algebra
$C^*(\Lambda)$ generated by a Cuntz--Krieger $\Lambda$-family. This class is
broad enough to contain irrational rotation algebras, Bunce--Deddens algebras
and a wide range of Kirchberg algebras \cite{PRRS}, yet combinatorial enough
that delicate structural properties can be read off the underlying graph.
Making that dictionary precise---between the combinatorics of $\Lambda$ and
the structure of $C^*(\Lambda)$---has been a central theme since
\cite{KP2,RSY1,KaP}.

A basic dividing line in the structure theory of nuclear $C^*$-algebras runs
between the stably finite and the purely infinite. For \emph{simple} $k$-graph
algebras, pure infiniteness is governed by aperiodicity and cofinality
\cite{S,PSS4}; in the non-simple setting, Evans and Sims \cite{ES} introduced
\emph{generalised cycles} and showed that their entrances control the presence
of infinite projections (see also \cite{HS,PRRS}). Pure infiniteness, in the
sense of Kirchberg and R\o rdam, is moreover the gateway to classification: a
separable, nuclear, strongly purely infinite algebra in the bootstrap class is
$\mathcal O_\infty$-stable \cite{KR}, of nuclear dimension one \cite{BGSW}, and
hence classified by ideal-related $KK$-theory. Our first aim is to give a clean,
graph-level characterisation of pure infiniteness for a row-finite, locally
convex $k$-graph with no sources, and to situate it within the gauge-invariant
ideal theory of \cite{KaP,PSS2}.

Our second aim---and the point we most wish to stress---is the resulting
computation of the \emph{nuclear dimension} of $C^*(\Lambda)$. Nuclear
dimension, introduced by Winter and Zacharias, is the central regularity
invariant of the Elliott classification programme, and its finiteness is one of
the equivalent forms of the Toms--Winter regularity conjecture. We show
(Proposition~\ref{prop:Oinfty-absorption}) that whenever $C^*(\Lambda)$ is purely
infinite of topological dimension zero---in particular whenever $H(\Lambda)$ is
finite---it is strongly purely infinite and of nuclear dimension one,
\emph{even when $C^*(\Lambda)$ is not simple}. This extends to the higher-rank,
non-simple setting the nuclear-dimension-one computation obtained for simple
UCT-Kirchberg $2$-graph algebras by Ruiz, Sims and S\o rensen \cite{RSS}; the
value one comes from the theorem of Bosa, Gabe, Sims and White that separable,
nuclear, $\mathcal O_\infty$-stable $C^*$-algebras have nuclear dimension one
\cite{BGSW} (finiteness of the nuclear dimension of strongly purely infinite
algebras being due to Szab\'o \cite{Sz}). What our combinatorial analysis
supplies is precisely the input those theorems need: a graph-level guarantee of
strong pure infiniteness, hence of $\mathcal O_\infty$-stability, for a large
class of non-simple $k$-graph algebras.

Beyond pure infiniteness we study the regularity properties that make these
algebras tractable for classification: real rank zero, extremal richness in the
sense of Brown and Pedersen \cite{BP,BP2}, the chain conditions (Noetherian and
Artinian) on the ideal lattice, stable rank one, and $\mathcal Z$- and
$\mathcal O_\infty$-stability. A recurring principle is that, once every ideal
of $C^*(\Lambda)$ is gauge invariant, each of these properties can be controlled
combinatorially through the lattice $H(\Lambda)$ of saturated hereditary subsets
of $\Lambda^0$ and the maximal tails of $\Lambda$.

\subsection*{Techniques}
Three tools underpin our arguments. The first is the gauge-invariant ideal
theory of \cite{RSY1,KaP}: for a row-finite, locally convex $k$-graph the map
$H\mapsto I_H$ is a lattice isomorphism from the saturated hereditary subsets
onto the gauge-invariant ideals, with quotients again of the form
$C^*(\Gamma(\Lambda\setminus H))$ (Theorem~\ref{thm:gauge-ideals}). The second is
strong aperiodicity: we show (Theorem~\ref{thm:gauge-invariant}) that it is
equivalent both to every generalised cycle in each maximal tail having an
entrance and to every ideal of $C^*(\Lambda)$ being gauge invariant, which
reduces the study of the entire ideal lattice to the combinatorics of
$H(\Lambda)$. The third is a systematic reduction to subquotients: inducting
along a maximal chain in $H(\Lambda)$, we assemble global properties of
$C^*(\Lambda)$ from those of the simpler algebras $C^*(\Gamma(\Lambda\setminus H))$
(Theorems~\ref{thm:RR0-finite-H} and~\ref{thm:approx}). Into this scheme we feed
the external structural inputs of Kirchberg--R\o rdam \cite{KR},
Pasnicu--R\o rdam \cite{PR} and Brown--Pedersen \cite{BP,BP2,BP3}.

\subsection*{Main results}
Our first main result, Theorem~\ref{thm:pure-inf}, characterises pure
infiniteness: for a row-finite, locally convex $k$-graph $\Lambda$, the algebra
$C^*(\Lambda)$ is purely infinite in the sense of Kirchberg--R\o rdam if and only
if every vertex projection is properly infinite; equivalently, every generalised
cycle in each maximal tail has an entrance and each vertex of each maximal tail
is connected to by a generalised cycle; equivalently again, $\Lambda$ is
strongly aperiodic and each maximal-tail vertex is reached by a generalised
cycle. The equivalence of the entrance condition with strong aperiodicity, and
with the gauge invariance of every ideal, is Theorem~\ref{thm:gauge-invariant}.

For real rank we prove (Theorem~\ref{thm:RR0-finite-H}) that if $\Lambda$ is
strongly aperiodic, each maximal-tail vertex is reached by a generalised cycle,
and $H(\Lambda)$ is finite, then $C^*(\Lambda)$ is purely infinite of
topological dimension zero, and has real rank zero precisely when it is
$K_0$-liftable; for $k=2$ the latter is the homological condition on the
connectivity matrices of \cite{PSS2}. Combining this with the permanence of
extremal richness, such algebras are extremally rich
(Theorems~\ref{thm:extreme-rich-structure} and~\ref{thm:RR0-extreme}).
Separately, when every primitive quotient is AF, $C^*(\Lambda)$ has stable
rank one (Theorem~\ref{thm:stable-rank}). Section~\ref{sec:chain} records the
accompanying chain conditions: topological dimension zero of
$\Prim(C^*(\Lambda))$ forces both the ascending and descending chain conditions
on ideals (Theorem~\ref{thm:chain-conditions}).

In Section~\ref{sec:trichotomy} we show that $k$-graph algebras populate all
three branches of the finite-ideal trichotomy: a Noetherian separable
$C^*$-algebra is purely infinite, stably finite, or neither, and these classes
are mutually exclusive (Theorem~\ref{thm:trichotomy}); we exhibit row-finite,
locally convex $k$-graphs realising each branch. From this analysis we deduce
two absorption theorems. When $C^*(\Lambda)$ is purely infinite and has
topological dimension zero---in particular whenever $H(\Lambda)$ is
finite---it is strongly purely infinite and $\mathcal O_\infty$-stable,
$C^*(\Lambda)\cong C^*(\Lambda)\otimes\mathcal O_\infty$
(Proposition~\ref{prop:Oinfty-absorption}); and under the chain conditions
$C^*(\Lambda)$ is $\mathcal Z$-stable (Corollary~\ref{cor:Z-stable}). In the
purely infinite, topological-dimension-zero case these algebras are moreover of
nuclear dimension one, even when they are not simple
(Proposition~\ref{prop:Oinfty-absorption}).

In the course of these arguments we clarify and, where necessary, correct the
hypotheses of several results in the graph and higher-rank graph literature,
flagging each emendation at the point where it occurs.

\subsection*{Comparison with earlier work of Pask, Sierakowski and Sims}
The three most closely related antecedents are the papers \cite{PSS2,PSS3,PSS4}
of the present author with Sierakowski and Sims, and it is worth making the
relationship precise.

The paper \cite{PSS3} treats the \emph{simple} case: for a $k$-graph whose
$C^*$-algebra is unital and simple it establishes a dichotomy between stable
finiteness and pure infiniteness, and gives necessary and sufficient
conditions, in terms of the underlying $k$-graph, for such an algebra to be
purely infinite. Theorem~\ref{thm:pure-inf} is the non-simple counterpart:
for an arbitrary row-finite, locally convex $k$-graph it characterises pure
infiniteness of $C^*(\Lambda)$ through the generalised cycles and maximal tails
of $\Lambda$, the simple case of \cite{PSS3} corresponding to
$H(\Lambda)=\{\emptyset,\Lambda^0\}$.

The paper \cite{PSS2} shows that strong aperiodicity is equivalent to
topological dimension zero of $C^*(\Lambda)$, and that a $C^*$-algebra of
topological dimension zero is purely infinite precisely when all of its vertex
projections are properly infinite; for purely infinite $2$-graph algebras it
characterises real rank zero as topological dimension zero together with a
homological condition on the adjacency matrices, and it exhibits strongly
purely infinite $2$-graph algebras of topological dimension zero that fail to
have real rank zero. Theorem~\ref{thm:pure-inf} sharpens the first of these
equivalences: since pure infiniteness already forces the entrance condition,
hence strong aperiodicity and thus topological dimension zero, the equivalence
between pure infiniteness and proper infiniteness of the vertex projections
holds with no separate dimension hypothesis, and is supplemented by graph-level
criteria valid for every $k$. Theorem~\ref{thm:RR0-finite-H} complements the
real-rank results of \cite{PSS2} in a different direction: it yields real rank
zero for all $k$, and for algebras with mixed AF and purely infinite
subquotients, avoiding the $2$-graph adjacency-matrix condition, at the cost of
assuming $H(\Lambda)$ finite. The examples of \cite{PSS2} show that this
reachability-and-finiteness hypothesis cannot simply be dropped.

Finally, \cite{PSS4} computes the stable rank of $C^*(\Lambda)$ for
\emph{finite} $k$-graphs that either contain no cycle with an entrance or are
cofinal, and determines exactly which finite locally convex $k$-graphs give
unital stably finite algebras, the stably finite ones being direct sums of
matrix algebras over tori of dimension at most $k$. We work throughout with
row-finite, locally convex $k$-graphs, so that infinite vertex sets are
permitted, and we use the local computations of
\cite[Proposition~2.7, Theorem~2.5]{PSS4} as inputs. Our stable-rank statement
(Theorem~\ref{thm:stable-rank}) is complementary rather than stronger: it
isolates the branch on which the stable rank equals one---namely when every
primitive quotient is AF---and does not reproduce the exact stable-rank
values, including those exceeding one and the value $\infty$ in the cofinal
purely infinite case, obtained in \cite{PSS4}.

\subsection*{Organisation}
Section~\ref{sec:background} fixes notation and recalls the gauge-invariant
ideal theory, cofinality and the primitive ideal space. Section~\ref{sec:pure-inf}
develops generalised cycles, strong aperiodicity, and the characterisation of
pure infiniteness. Section~\ref{sec:real-rank} treats real rank zero and the
essential subgraph. Section~\ref{sec:extreme} develops extremal richness, with
worked AF, purely infinite and mixed examples. Section~\ref{sec:chain} collects
the chain conditions, stable rank and $K$-theory consequences, and
Section~\ref{sec:trichotomy} establishes the trichotomy and the absorption
theorems, including that the purely infinite algebras of topological dimension
zero have nuclear dimension one, even when they are not simple
(Proposition~\ref{prop:Oinfty-absorption}).

\section{Background}\label{sec:background}

Let $\Prim(A)$ denote the set of primitive ideals in a $C^*$-algebra $A$. For each
ideal $I\subset A$, the \emph{hull} of $I$ is

\[
\mathrm{hull}(I) := \{ P\in\Prim(A) : I \subset P\},
\]

and for each subset $S\subset\Prim(A)$, the \emph{kernel} of $S$ is

\[
\ker(S) := \bigcap_{P\in S} P.
\]

The Jacobson topology on $\Prim(A)$ is determined by the closure operation

\[
\overline{S} := \mathrm{hull}(\ker(S)),\qquad S\subset\Prim(A).
\]

\begin{definition}
A $C^*$-algebra $A$ has \emph{topological dimension zero} if $\Prim(A)$, endowed
with the Jacobson topology, has a basis of compact open sets.
\end{definition}

If $A$ has no unit, we write $\tilde{A}$ for its unitisation, and $A_+$ for the
positive elements of $A$.

\subsection{Higher rank graphs}

For $k\geq 0$, a \emph{$k$-graph} (or rank-$k$ graph) is a nonempty countable small category $\Lambda$
equipped with a functor $d:\Lambda\to\N^k$ satisfying the \emph{factorisation property}:
for all $\lambda\in\Lambda$ and $m,n\in\N^k$ with $d(\lambda)=m+n$ there exist
unique $\mu,\nu\in\Lambda$ such that $d(\mu)=m$, $d(\nu)=n$ and $\lambda=\mu\nu$.
We write $\Lambda^n := d^{-1}(n)$ and $\Lambda^0$ for the vertices.

For $m,n\in\N^k$ we write $m\vee n$ for the coordinatewise maximum, and define
$m\leq n$ if $m_i\leq n_i$ for all $i$.

\begin{definition}
Let $\Lambda$ be a $k$-graph.
\begin{itemize}
\item $\Lambda$ is \emph{row-finite} if $v\Lambda^n$ is finite for each
      $v\in\Lambda^0$ and $n\in\N^k$.
\item $\Lambda$ has \emph{no sources} if $v\Lambda^{\varepsilon_i}\neq\emptyset$ for all
      $v\in\Lambda^0$ and $i\in\{1,\dots,k\}$.
\item $\Lambda$ is \emph{finite} if $\Lambda^0$ is finite and $\Lambda$ is row-finite.
\item $\Lambda$ is \emph{strongly connected} if $u\Lambda v\neq\emptyset$ for all
      $u,v\in\Lambda^0$.
\end{itemize}
\end{definition}

For $\lambda\in\Lambda$ and $0\leq m\leq n\leq d(\lambda)$ we write
$\lambda(m,n)$ for the unique subpath with $d(\lambda(m,n))=n-m$ and
$\lambda=\lambda(0,m)\lambda(m,n)\lambda(n,d(\lambda))$. We write
$\lambda(n):=\lambda(n,n)=s(\lambda(0,n))$.

\begin{definition}[Minimal common extensions]
Let $\Lambda$ be a $k$-graph. For $\mu,\nu\in\Lambda$ with $r(\mu)=r(\nu)$ we write
\[
\mathrm{MCE}(\mu,\nu):=\mu\Lambda\cap\nu\Lambda\cap\Lambda^{d(\mu)\vee d(\nu)}
=\{\gamma\in\Lambda^{d(\mu)\vee d(\nu)}:\gamma(0,d(\mu))=\mu,\ \gamma(0,d(\nu))=\nu\}
\]
for the set of \emph{minimal common extensions} of $\mu$ and $\nu$. 
\end{definition}

\noindent
The following definition first appears in \cite[Definition 3.9]{RSY1}.

\begin{definition}[Locally convex]
A $k$-graph $\Lambda$ is \textit{locally convex} if for all distinct
$i,j\in\{1,\dots,k\}$ and all $e \in \Lambda^{\varepsilon_i}$ and
$f \in \Lambda^{\varepsilon_j}$ with $r(e) = r(f)$, the sets
$s(e)\Lambda^{\varepsilon_j}$ and $s(f)\Lambda^{\varepsilon_i}$ are nonempty.
For a locally convex $k$-graph $\Lambda$ and 
$n \in \mathbb{N}^k$, we write 
\[
\Lambda^{\leq n} := \{ \lambda \in \Lambda : d(\lambda) \leq n 
\text{ and } s ( \lambda ) \Lambda^{\varepsilon_i} = \emptyset 
\text{ whenever } d  (\lambda ) + \varepsilon_i \leq n \} .
\]
\end{definition}

\noindent
We will often illustrate $k$-graphs as $k$-coloured graphs. We refer to \cite{HRSW}
for the details, but in short there is a one-to-one correspondence between $k$-graphs and
$k$-coloured graphs together with factorisation rules for bi-coloured paths of length~2 satisfying
an associativity condition \cite[Equation~(3.2)]{HRSW}.

\begin{example}[One-vertex rank-two graph]\label{ex:rank2}
Fix $m_1,m_2\geq 2$ and sets $E_1=\{f_1,\dots,f_{m_1}\}$,
$E_2=\{g_1,\dots,g_{m_2}\}$ of degree-$\varepsilon_1$ and degree-$\varepsilon_2$ generators.
Let $\theta:\{1,\dots,m_1\}\times\{1,\dots,m_2\}\to
\{1,\dots,m_2\}\times\{1,\dots,m_1\}$ be a bijection. The monoidal
$2$-graph $F_\theta$ has one vertex, generators $E_1\cup E_2$,
and relations

\[
f_i g_j = g_{j'} f_{i'}\quad\text{whenever}\quad
\theta(i,j)=(j',i').
\]

\end{example}

\begin{example}[A model two-dimensional graph]\label{ex:skew-tikz}
Following \cite{KP2},  let $\Omega_2$ be the category with objects $\Obj \, ( \Omega_2 ) = \mathbb{N}^2$, and
morphisms  $\Omega_2  = \{ (m,n) \in \mathbb{N}^2 \times \mathbb{N}^2 : m \le n \}$;  the range and source maps are given by $r ( m , n ) = m$, $s ( m , n ) = n$ so
that $( \ell,n) = ( \ell,m) (m,n)$.  Let $d : \Omega_k \rightarrow \mathbb{ N}^2$ be defined by $d ( m , n ) = n - m$. It is then straightforward to show that $\Omega_2$ is a row-finite, locally convex $2$-graph. The vertices $\Omega_2^0 = \{ (m,m) : m \in \mathbb{N}^2  \}$
may be identified with $\mathbb{N}^2$, since $(m,m) (m,n) = (m,n) =  (m,n)(n,n)$
for all $m \le n \in \mathbb{N}^2$.
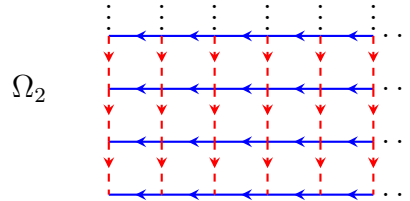
\[
\begin{tikzpicture}[scale=0.7]

\draw[blue, thick,midarrow] (1,0)--(0,0);
\draw[blue, thick,midarrow] (2,0)--(1,0);
\draw[blue, thick,midarrow] (3,0)--(2,0);
\draw[blue, thick,midarrow] (4,0)--(3,0);
\draw[blue,thick,midarrow] (5,0)--(4,0);

\draw[blue, thick,midarrow] (1,1)--(0,1);
\draw[blue, thick,midarrow] (2,1)--(1,1);
\draw[blue, thick,midarrow] (3,1)--(2,1);
\draw[blue, thick,midarrow] (4,1)--(3,1);
\draw[blue,thick,midarrow] (5,1)--(4,1);

\draw[blue, thick,midarrow] (1,2)--(0,2);
\draw[blue, thick,midarrow] (2,2)--(1,2);
\draw[blue, thick,midarrow] (3,2)--(2,2);
\draw[blue, thick,midarrow] (4,2)--(3,2);
\draw[blue,thick,midarrow] (5,2)--(4,2);

\draw[blue, thick,midarrow] (1,3)--(0,3);
\draw[blue, thick,midarrow] (2,3)--(1,3);
\draw[blue, thick,midarrow] (3,3)--(2,3);
\draw[blue, thick,midarrow] (4,3)--(3,3);
\draw[blue,thick,midarrow] (5,3)--(4,3);

\draw[red, dashed, thick,midarrow] (0,1)--(0,0);
\draw[red, dashed, thick,midarrow] (0,2)--(0,1);
\draw[red, dashed, thick,midarrow] (0,3)--(0,2);

\draw[red, dashed, thick,midarrow] (1,1)--(1,0);
\draw[red, dashed, thick,midarrow] (1,2)--(1,1);
\draw[red, dashed, thick,midarrow] (1,3)--(1,2);

\draw[red, dashed, thick,midarrow] (2,1)--(2,0);
\draw[red, dashed, thick,midarrow] (2,2)--(2,1);
\draw[red, dashed, thick,midarrow] (2,3)--(2,2);

\draw[red, dashed, thick,midarrow] (3,1)--(3,0);
\draw[red, dashed, thick,midarrow] (3,2)--(3,1);
\draw[red, dashed, thick,midarrow] (3,3)--(3,2);

\draw[red, dashed, thick,midarrow] (4,1)--(4,0);
\draw[red, dashed, thick,midarrow] (4,2)--(4,1);
\draw[red, dashed, thick,midarrow] (4,3)--(4,2);

\draw[red, dashed, thick,midarrow] (5,1)--(5,0);
\draw[red, dashed, thick,midarrow] (5,2)--(5,1);
\draw[red, dashed, thick,midarrow] (5,3)--(5,2);

\node at (-1.5,2) {$\Omega_{2}$};
\node at (5.5,0) {$\ldots$};
\node at (5.5,1) {$\ldots$};
\node at (5.5,2) {$\ldots$};
\node at (5.5,3) {$\ldots$};
\node at (0,3.5) {$\vdots$};
\node at (1,3.5) {$\vdots$};
\node at (2,3.5) {$\vdots$};
\node at (3,3.5) {$\vdots$};
\node at (4,3.5) {$\vdots$};
\node at (5,3.5) {$\vdots$};

\end{tikzpicture}
\]
\end{example}

\subsection{Higher rank graph $C^*$-algebras}

Let $\Lambda$ be a locally convex, row-finite $k$-graph with no sources.

\begin{definition}[Cuntz--Krieger family]
A \emph{Cuntz--Krieger $\Lambda$-family} in a $C^*$-algebra $B$ is a map
$s:\Lambda\to B$ such that:
\begin{enumerate}
\item $\{s_v : v\in\Lambda^0\}$ is a family of mutually orthogonal projections;
\item $s_\mu s_\nu = s_{\mu\nu}$ whenever $s(\mu)=r(\nu)$;
\item $s_\lambda^* s_\lambda = s_{s(\lambda)}$ for all $\lambda\in\Lambda$;
\item For all $v\in\Lambda^0$ and $n\in\N^k$,

\[
s_v = \sum_{\lambda\in v\Lambda^n} s_\lambda s_\lambda^*.
\]

\end{enumerate}
\end{definition}

\begin{definition}
The \emph{higher rank graph $C^*$-algebra} $C^*(\Lambda)$ is the universal
$C^*$-algebra generated by a Cuntz--Krieger $\Lambda$-family.
\end{definition}

\begin{remark}[Density]\label{rem:dense-span}
The linear span of $\{s_\mu s_\nu^* : s(\mu)=s(\nu)\}$ is a dense
$*$-subalgebra of $C^*(\Lambda)$; in particular $C^*(\Lambda)$ is separable.
\end{remark}

\subsection{Saturated hereditary subsets and gauge-invariant ideals}

\begin{definition}[Saturated and hereditary]\label{def:sat-her}
Let $\Lambda$ be a row-finite $k$-graph with no sources. A subset
$H\subset\Lambda^0$ is:
\begin{itemize}
\item \emph{hereditary} if $v\in H$ and $v\Lambda w\neq\emptyset$ imply $w\in H$;
\item \emph{saturated} if for all $v\in\Lambda^0$ and $n\in\N^k$,

\[
s(v\Lambda^n)\subset H \;\Rightarrow\; v\in H.
\]

\end{itemize}
\end{definition}

For a hereditary $H\subset\Lambda^0$ we write $I_H$ for the ideal in
$C^*(\Lambda)$ generated by $\{s_v : v\in H\}$, and for an ideal
$I\subset C^*(\Lambda)$ we set $H_I := \{ v\in\Lambda^0 : s_v\in I\}$.

We shall refer to this result frequently, the proof may be found in Theorem 5.5]{S} or \cite[Theorem 5.2]{RSY1}.

\begin{theorem}[Gauge-invariant ideals]\label{thm:gauge-ideals}
Let $(\Lambda,d)$ be a row-finite, locally convex $k$-graph. Then:
\begin{enumerate}
\item If $I$ is a nonzero gauge-invariant ideal of $C^*(\Lambda)$, then
      $H_I$ is a nonempty saturated hereditary subset and $I=I_{H_I}$.
\item If $H\subset\Lambda^0$ is saturated hereditary, then $I_H$ is
      gauge invariant and $H=H_{I_H}$.
\item The map $H\mapsto I_H$ is a lattice isomorphism from the lattice
      of saturated hereditary subsets of $\Lambda^0$ onto the lattice
      of closed gauge-invariant ideals of $C^*(\Lambda)$.
\item If $H\subset\Lambda^0$ is saturated hereditary, then the quotient
      $C^*(\Lambda)/I_H$ is canonically isomorphic to $C^*(\Gamma(\Lambda\setminus H))$,
      where

\[
      \Gamma(\Lambda\setminus H) := \big(\Lambda^0\setminus H,\,
      \{\lambda\in\Lambda : s(\lambda)\notin H\},\,r,s\big).
      \]

\end{enumerate}
\end{theorem}

\subsection{Cofinality}

\begin{definition}[Cofinal]\label{def:cofinal}
Let $\Lambda$ be a row-finite $k$-graph with no sources. We say
$\Lambda$ is \emph{cofinal} if for all $v,w\in\Lambda^0$ there exists
$m\in\N^k$ such that $v\Lambda s(\lambda)\neq\emptyset$ for all
$\lambda\in w\Lambda^m$.
\end{definition}

\begin{lemma}[Cofinality and saturated hereditary subsets]\label{lem:cofinal}
Let $\Lambda$ be a row-finite, locally convex $k$-graph. Then $\Lambda$
is cofinal if and only if the lattice $H(\Lambda)$ of saturated hereditary
subsets is $\{ \emptyset,\Lambda^0\}$.
\end{lemma}

\section{Purely infinite higher rank graph $C^*$-algebras}\label{sec:pure-inf}

\subsection{Infinite and properly infinite elements}

Let $A$ be a $C^*$-algebra and $a,b\in A_+$.

We write $a\precsim b$ if there exists a sequence $(x_n)$ in $A$ such that
$x_n^* b x_n\to a$ in norm. A positive element $a\in A$ is \emph{infinite}
if there exists a nonzero $b\in A_+$ such that $a\oplus b\precsim a$.
If $a$ is nonzero and $a\oplus a\precsim a$, then $a$ is \emph{properly infinite}.

\begin{definition}[Purely infinite]\label{def:pure-inf}
A $C^*$-algebra $A$ is \emph{purely infinite} (in the sense of
Kirchberg--R\o rdam) if it has no characters and if for every pair
$x,y\in A_+$ with $y$ in the closed two-sided ideal generated by $x$,
there exists a sequence $(a_n)$ in $A$ such that $a_n x a_n^*\to y$.
\end{definition}

\subsection{Loops with entrances and generalised cycles}

\noindent
This definition comes from \cite[Definition 8.7]{S}.

\begin{definition}[Loop with entrance]\label{def:loop-entrance}
Let $\Lambda$ be a $k$-graph. A path $\mu\in\Lambda$ is a \emph{loop}
if $s(\mu)=r(\mu)$. It is a \emph{loop with an entrance} if there exists
$\alpha\in s(\mu)\Lambda$ with $d(\mu)\geq d(\alpha)$ and
$\mu(0,d(\alpha))\neq\alpha$.
\end{definition}

\noindent
The definition was first given in \cite[Definition 3.1]{ES}.

\begin{definition}[Generalised cycle]\label{def:gen-cycle}
Let $\Lambda$ be a locally convex $k$-graph. A \emph{generalised cycle}
is a pair $(\mu,\nu)\in\Lambda\times\Lambda$ such that:
\begin{itemize}
\item $\mu\neq\nu$;
\item $s(\mu)=s(\nu)$ and $r(\mu)=r(\nu)$;
\item $\mathrm{MCE}(\mu\alpha,\nu)\neq\emptyset$ for all
      $\alpha\in s(\mu)\Lambda$.
\end{itemize}
An \emph{entrance} to $(\mu,\nu)$ is a path $\alpha\in s(r(\mu))\Lambda$
such that $\mathrm{MCE}(r(\mu)\alpha,\mu)=\emptyset$.
\end{definition}

\subsection{Strong aperiodicity and gauge-invariant ideals}

Here we shall be using the definition of aperiodicity which use notion of the boundary paths of a locally convex row-finite higher rank graph.

\begin{definition}[Boundary path]
For each $m \in (\mathbb{N} \cup \{\infty\})^k$, and $k \ge 1$ we define a $k$-graph
$\Omega_{k,m}$ by $\Omega_{k,m} = \{(p,q) \in \mathbb{N}^k \times \mathbb{N}^k : p \le q \le m\}$ with range map
$r(p,q) = (p,p)$, source map $s(p,q) = (q,q)$, and degree map $d(p,q) = q-p$. We identify $\Omega_{k,m}^0$ with $\{p \in \mathbb{N}^k : p \le m\}$ via the map $(p,p) \mapsto p$. Given a $k$-graph
and $m \in \mathbb{N}^k$ there is a bijection from $\Lambda^m$ to the set of morphisms $x : \Omega_{k,m}
\to \Lambda$, given by $\lambda \mapsto \big((p,q)\mapsto \lambda(p,q)\big)$; the inverse is the
map $x \mapsto x(0,m)$. Thus, for each $m\in \mathbb{N}^k$ we may identify the collection of $k$-graph
morphisms from $\Omega_{k,m}$ to $\Lambda$ with $\Lambda^m$. We extend this idea beyond $m\in
\mathbb{N}^k$ as follows: Given $m \in (\mathbb{N} \cup \{\infty\})^k \backslash \mathbb{N}^k$, we regard each $k$-graph
morphism $x : \Omega_{k,m} \to \Lambda$ as a path of degree $m$ in $\Lambda$ and write $d(x)= m$ and $r(x)= x(0)$; we denote the set of all such paths by $\Lambda^m$. When $m=( \infty , \ldots , \infty )$, we denote $\Omega_{k,m}$ by $\Omega_k$ and we call a path of degree $m$ an \emph{infinite path}. We denote by
$W_\Lambda$ the collection $\bigcup_{m \in (\mathbb{N} \cup \{\infty\})^k} \Lambda^m$ of all paths in
$\Lambda$; our conventions allow us to regard $\Lambda$ as a subset of $W_\Lambda$. We call
$W_\Lambda$ the \emph{path space} of $\Lambda$.  If $\Lambda$ has no sources then $W_\Lambda=\Lambda$

We set
$$\Lambda^{\leq \infty}= \big\{ x\in W_\Lambda: x(n)\Lambda^{\varepsilon_i}=\emptyset\ \textrm{whenever}\  n\leq d(x)\ \textrm{and} \  n_i = d(x)_i\big\},
$$
and for $v\in \Lambda^0$, we define $v\Lambda^{\leq \infty}=\{x\in \Lambda^{\leq \infty}:
r(x)=v\}$. We call $\Lambda^{\leq \infty}$ the \emph{boundary path space} of $\Lambda$. If $\Lambda$ has no sources then $\Lambda^{\leq \infty}=\Lambda^\infty$ as defined in \cite[Definition 2.1]{KP2}.
For $n\in \mathbb{N}^k$ there is a \emph{shift map} $\sigma^n: \{x \in W_\Lambda :
n\leq d(x) \} \to W_\Lambda$ such that $d(\sigma^n(x))=d(x)-n$ and $\sigma^n(x)(p,q)=x(n+p,n+q)$
for $0\leq p \leq q \leq d(x)-n$ where we use the convention $\infty-a=\infty$ for $a \in \mathbb{N}$.
For $x \in W_\Lambda$ and $n\leq d(x)$, we then have $x(0,n)\sigma^n(x) = x$.
\end{definition}

\noindent
The following definition is taken from \cite[Lemma 3.2 (iv)]{RS1} (see also \cite[Appendix A]{RS1}. It can be improved slightly (see \cite[Definition 4.1]{LP}) but it is only a minor improvement).

\begin{definition}[Local periodicity and aperiodicity]\label{def:local-per}
Let $(\Lambda,d)$ be a row-finite $k$-graph. Fix $v\in\Lambda^0$ and
$m\neq n\in\N^k$. We say $\Lambda$ has \emph{local periodicity $(m,n)$
at $v$} if for every boundary path $x\in v \Lambda^{\le \infty}$ we have

\[
m\wedge d(x) = n\wedge d(x)\quad\text{and}\quad
\sigma^{m\wedge d(x)}(x) = \sigma^{n\wedge d(x)}(x).
\]

\noindent
We say $\Lambda$ has \emph{no local periodicity} if it does not have
local periodicity $(m,n)$ at any $v\in\Lambda^0$ and $m\neq n\in\N^k$.
\end{definition}

\begin{definition}[Aperiodicity and strong aperiodicity]\label{def:aperiodic}
Let $\Lambda$ be a row-finite $k$-graph with no sources.
We say $\Lambda$ is \emph{aperiodic} if every vertex $v\in\Lambda^0$
has no local periodicity. We say $\Lambda$ is \emph{strongly aperiodic}
if $\Gamma (\Lambda\setminus H)$ is aperiodic for every saturated hereditary
$H\subset\Lambda^0$.
\end{definition}

\noindent
From \cite[Theorem 3.2, Corollary 3.9]{PSS2} (see also \cite[Theorem 7.2]{S}) we have:

\begin{theorem}[Gauge-invariant ideals and strong aperiodicity]\label{thm:gauge-invariant}
Let $\Lambda$ be a locally convex, row-finite $k$-graph. The following
are equivalent:
\begin{enumerate}
\item $\Lambda$ is strongly aperiodic;
\item every generalised cycle in each maximal tail $M$ has an entrance
      in $M$;
\item every ideal of $C^*(\Lambda)$ is gauge invariant.
\end{enumerate}
\end{theorem}

\noindent
For each subset $H \in \mathcal{H} ( \Lambda )$, let $I_H$ denote the ideal in $C^* (\Lambda )$ generated by $\{ s_v : v \in H \}$. Then it can be shown that for $H \in \mathcal{H} ( \Lambda )$,
\[
I_H = span \{s_\alpha  s_\beta^* : \alpha, \beta \in \Lambda \text{ and } s(\alpha)= s( \beta ) \in  H \}.
\]

\noindent
In particular, $I_H$ is gauge invariant in the sense that $\gamma_z(a)= a$ for all $a \in I_H$ and $z \in \mathbb{T}^k$.

\subsection{Pure infiniteness}

\begin{definition}[Maximal tail]
Following \cite{KaP} we recall the definition of a maximal tail. Let $\Lambda$ be a
row-finite $k$-graph with no sources. A nonempty subset $T$ of $\Lambda^0$ is called a
\emph{maximal tail} if
\begin{itemize}
\item[(MT1)] for all $v_1,v_2\in T$ there exists $w\in T$ such that $v_1 \Lambda w \neq \emptyset$ and $v_2\Lambda w\neq \emptyset$,
\item[(MT2)] for every $v\in T$ and $1\leq i\leq k$ there exist $e \in v \Lambda^{\varepsilon_i}$ such that $s(e) \in T$, and
\item[(MT3)] for all $w\in T$ and $v \in \Lambda^0$ with $v \Lambda w \neq \emptyset$ we  have $v\in T$.
\end{itemize}
\end{definition}

\noindent
Let $\mathcal{M} ( \Lambda )$ denote the collection of maximal tails in $\Lambda$, then $\Lambda^0 \in \mathcal{M} ( \Lambda )$.

The following is a generalisation of \cite[Theorem 2.3]{HS} in the setting of higher rank graphs and their $C^*$-algebras. New arguments are needed due to the lack of an analogue of \cite[Corollary 2.8]{KPR} (``no loops$\,\Rightarrow\,$AF'') for higher rank graphs  and an error in (a)$\Rightarrow$(e) in \cite[Theorem 2.3]{HS}.
\begin{theorem}[Pure infiniteness]\label{thm:pure-inf}
Let $\Lambda$ be a locally convex, row-finite $k$-graph, and consider the
following conditions:
\begin{enumerate}
    \item\label{pi:a}
    Every nonzero hereditary $C^*$-subalgebra in every quotient of
    $C^*(\Lambda)$ contains an infinite projection.

    \item\label{pi:b}
    $C^*(\Lambda)$ is purely infinite in the sense of Kirchberg--R\o rdam.

    \item\label{pi:c}
    For each vertex $v\in\Lambda^0$, the projection $s_v$ is properly
    infinite in $C^*(\Lambda)$.

    \item\label{pi:d}
    All generalised cycles in each maximal tail $M$ have entrances in $M$,
    and each vertex in each maximal tail $M$ is connected to by a
    generalised cycle in $M$.

    \item\label{pi:e}
    The $k$-graph $\Lambda$ is strongly aperiodic, and each vertex in each
    maximal tail $M$ is connected to by a generalised cycle in $M$.
\end{enumerate}
Then \textup{(\ref{pi:d})} and \textup{(\ref{pi:e})} are equivalent, and
\[
\textup{(\ref{pi:d})}\Longrightarrow\textup{(\ref{pi:a})}
\Longrightarrow\textup{(\ref{pi:b})}\Longrightarrow\textup{(\ref{pi:c})}.
\]
If in addition $\Lambda^0$ is finite, then
\textup{(\ref{pi:c})}$\Rightarrow$\textup{(\ref{pi:d})}, so that all five
conditions are equivalent.
\end{theorem}

\begin{proof}
(\ref{pi:a})$\Rightarrow$(\ref{pi:b}) follows from \cite[Proposition~4.7]{KR}.
(\ref{pi:b})$\Rightarrow$(\ref{pi:c}) follows from \cite[Proposition~4.16]{KR}.

\smallskip

(\ref{pi:c})$\Rightarrow$(\ref{pi:d}).  
Let $M$ be a maximal tail. By \cite[Theorem~3.12]{KaP}, the complement
$\Lambda^0\setminus M$ is saturated hereditary, and
$C^*(\Lambda)/I_M\cong C^*(\Gamma(M))$ by \cite[Theorem~5.2]{RSY1}.  

Suppose first that $(\mu,\nu)$ is a generalised cycle in $M$ with no
entrance. Then the element $V := s_\mu s_\nu^*$ is a unitary with full
spectrum in $s_{s(\mu)}C^*(\Gamma(M))s_{s(\mu)}$ \cite[Proposition~3.10]{ES}.
The ideal generated by $V$ is Morita equivalent to $C(\mathbb T)$
\cite[Lemma~3.9]{PRRS}, so $s_{s(\mu)}$ is not properly infinite in
$C^*(\Gamma(M))$, contradicting~(\ref{pi:c}).

Next suppose that $\Lambda^0$ is finite. By the previous paragraph we may
assume that every generalised cycle in $M$ has an entrance, whence
$\Gamma(M)$ is cofinal and aperiodic and $C^*(\Gamma(M))$ is simple
\cite{RS1}. If $M$ contained no generalised cycle, then, since $\Gamma(M)$
has finitely many vertices, $C^*(\Gamma(M))$ would be AF by
\cite[Proposition~5.4]{ES}, and every $s_v$ would be finite,
contradicting~(\ref{pi:c}). Hence $M$ contains a generalised cycle, and by
cofinality of $\Gamma(M)$ every vertex of $M$ is connected to by such a
cycle. Thus, when $\Lambda^0$ is finite, (\ref{pi:d}) holds.

\smallskip

(\ref{pi:d})$\Rightarrow$(\ref{pi:e}).  
By Theorem~\ref{thm:gauge-invariant}, the entrance condition for
generalised cycles in each maximal tail is equivalent to strong
aperiodicity.

\smallskip

(\ref{pi:e})$\Rightarrow$(\ref{pi:a}).  
Since $\Lambda$ is strongly aperiodic, every ideal of $C^*(\Lambda)$ is
gauge invariant by \cite[Theorem~7.2]{S}. Hence each ideal has the form
$I_H$ for a saturated hereditary $H\subset\Lambda^0$, and

\[
C^*(\Lambda)/I_H \;\cong\; C^*(\Gamma(\Lambda\setminus H))
\]

\noindent
by \cite[Theorem~5.2]{RSY1}.  

By hypothesis, every generalised cycle in $\Gamma(\Lambda\setminus H)$ has an
entrance, and every vertex is connected to by such a cycle. By
\cite[Proposition~8.8]{S}, every hereditary subalgebra of
$C^*(\Gamma(\Lambda\setminus H))$ contains an infinite projection. This
verifies~(\ref{pi:a}).
\end{proof}

\begin{remark}\label{rem:pi-finite-vertices}
Finiteness of $\Lambda^0$ is needed only for the implication
\textup{(\ref{pi:c})}$\Rightarrow$\textup{(\ref{pi:d})}, which rests on the fact
that a $k$-graph with finitely many vertices and no generalised cycle has AF
$C^*$-algebra \cite[Proposition~5.4]{ES}. For infinitely many vertices the
analogous statement is open: the simple, aperiodic and cofinal $2$-graph
$\Lambda_{\mathrm{II}}$ of \cite[\S6]{ES} contains no generalised cycle, yet
whether $C^*(\Lambda_{\mathrm{II}})$ is AF is left open in
\cite[Remark~6.13]{ES}. The applications of Theorem~\ref{thm:pure-inf} below use
only the unconditional implication
\textup{(\ref{pi:d})}$\Rightarrow$\textup{(\ref{pi:a})}.
\end{remark}

\begin{proposition}\label{prop:properly-infinite-vertices}
Let \(\Lambda\) be a row-finite, locally convex \(k\)-graph with no sources.
Suppose that every vertex \(v\in\Lambda^0\) can be reached from a cycle
with an entrance. Then each vertex projection \(s_v\) in \(C^*(\Lambda)\)
is properly infinite.
\end{proposition}

\begin{proof}
Let \(v\in\Lambda^0\). By hypothesis, there exists a cycle \(\mu\) with an
entrance \(\alpha\) such that \(v\Lambda s(\mu)\neq\emptyset\). Let
\(p=s_{s(\mu)}\). Since \(\mu\) is a cycle with an entrance, the standard
argument of \cite[Proposition~2.7]{PSS4} shows that \(p\) is properly
infinite.

Since \(v\Lambda s(\mu)\neq\emptyset\), choose a path \(\lambda\) from \(v\) to
\(s(\mu)\). Then

\[
s_v \ge s_\lambda s_\lambda^* \sim s_{s(\mu)} = p,
\]

so \(s_v\) dominates a properly infinite projection. Proper infiniteness
passes to dominated projections, hence \(s_v\) is properly infinite.
\end{proof}

\begin{examples}
\begin{enumerate}
\item If $\Lambda$ is strongly aperiodic and $C^* ( \Lambda )$ has an AF quotient, or a quotient Morita equivalent to $C ( \mathbb{T} )$ then $C^* ( \Lambda )$ is not purely infinite. Indeed, pure infiniteness in the sense of Kirchberg--R\o rdam passes to quotients and is preserved under Morita equivalence \cite{KR}. A nonzero AF algebra and $C(\mathbb{T})$ are stably finite---each carries a faithful tracial state---and hence neither is purely infinite. So if $C^*(\Lambda)$ were purely infinite, then so would be any quotient of it, and so would any algebra Morita equivalent to such a quotient; either possibility contradicts the hypothesis. (Strong aperiodicity plays no role in this implication.)
\item The $1$-graph $\Gamma$ depicted below which corresponds to the directed graph shown in \cite[Example 3.10]{KPR} is not purely infinite (in the sense of \cite{KR}) as it does not satisfy the strong aperiodicity condition and so Theorem~\ref{thm:pure-inf} does not apply. 
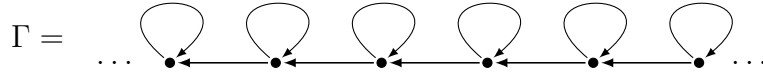
\[
\begin{tikzpicture}[scale=0.35]
\node at (-5,1) {$\Gamma=$};
\node at (-2,0) {$\dots$};
\node at (22,0) {$\dots$};
\node[circle,inner sep=0pt] (p11) at (0, 0)
{\begin{tikzpicture}[scale=0.4]
\node at (0.9, 0.8) [draw, fill=black] {$.$};\end{tikzpicture}}
edge[-latex, loop, out=140, in=40, min distance=120, looseness=2.2] (p11);
\node[circle,inner sep=0pt] (p12) at (4, 0)
{\begin{tikzpicture}[scale=0.4]
\node at (0.9, 0.8) [draw, fill=black] {$.$};
\end{tikzpicture}}
edge[-latex, loop, out=140, in=40, min distance=120, looseness=2.2] (p12);
\node[circle,inner sep=0pt] (p13) at (8, 0)
{\begin{tikzpicture}[scale=0.4]
\node at (0.9, 0.8) [draw, fill=black] {$.$};
\end{tikzpicture}}
edge[-latex, loop, out=140, in=40, min distance=120, looseness=2.2] (p13);
\node[circle,inner sep=0pt] (p14) at (12, 0)
{\begin{tikzpicture}[scale=0.4]
\node at (0.9, 0.8) [draw, fill=black] {$.$};
\end{tikzpicture}}
edge[-latex, loop, out=140, in=40, min distance=120, looseness=2.2] (p14);
\node[circle,inner sep=0pt] (p15) at (16, 0)
{\begin{tikzpicture}[scale=0.4]
\node at (0.9, 0.8) [draw, fill=black] {$.$};
\end{tikzpicture}}
edge[-latex, loop, out=140, in=40, min distance=120, looseness=2.2] (p15);
\node[circle,inner sep=0pt] (p16) at (20, 0)
{\begin{tikzpicture}[scale=0.4]
\node at (0.9, 0.8) [draw, fill=black] {$.$};
\end{tikzpicture}}
edge[-latex, loop, out=140, in=40, min distance=120, looseness=2.2] (p16);
\draw[style=semithick, -latex] (p12.west)--(p11.east);
\draw[style=semithick, -latex] (p13.west)--(p12.east);
\draw[style=semithick, -latex] (p14.west)--(p13.east);
\draw[style=semithick, -latex] (p15.west)--(p14.east);
\draw[style=semithick, -latex] (p16.west)--(p15.east);
\end{tikzpicture}
\]

\noindent
Furthermore, it does not satisfy the Cuntz version of simple pure infiniteness because it is not simple. Hence \cite[Theorem 3.9]{KPR} is false as stated as it applies to non-simple algebras.
See also \cite[Proposition 4.9]{KP2} for a similar misstatement; the original
version given in \cite[Proposition 4.9]{KP2} is incorrect, see the erratum
\cite{KPSerr}.
Thankfully, the dichotomy \cite[Corollary 3.11]{KPR} is still true as there is a simplicity hypothesis.
\item The $1$-graph $\Omega$ depicted below corresponds to the directed graph with the arrows reversed which clearly satisfies condition (K) (in the sense of \cite{KPRR}) and so $\Omega$ is strongly aperiodic by \cite[Lemma 3.2]{KaP}.
\[
\begin{tikzpicture}[scale=0.35]
\node at (-5,1) {$\Omega=$};
\node at (-2,0) {$\dots$};
\node at (22,0) {$\dots$};
\node[circle,inner sep=0pt] (p11) at (0, 0)
{\begin{tikzpicture}[scale=0.4]
\node at (0.9, 0.8) {\tiny $v_{-2}$};\end{tikzpicture}}
edge[-latex, loop, out=120, in=60, min distance=50, looseness=2] (p11)
edge[-latex, loop, out=140, in=40, min distance=110, looseness=2.2] (p11);
\node[circle,inner sep=0pt] (p12) at (4, 0)
{\begin{tikzpicture}[scale=0.4]
\node at (0.9, 0.8) {\tiny $v_{-1}$};
\end{tikzpicture}}
edge[-latex, loop, out=120, in=60, min distance=50, looseness=2] (p12)
edge[-latex, loop, out=140, in=40, min distance=110, looseness=2.2] (p12);
\node[circle,inner sep=0pt] (p13) at (8, 0)
{\begin{tikzpicture}[scale=0.4]
\node at (0.9, 0.8) {\tiny $v_0$};
\end{tikzpicture}}
edge[-latex, loop, out=120, in=60, min distance=60, looseness=2] (p13)
edge[-latex, loop, out=140, in=40, min distance=120, looseness=2.2] (p13);
\node[circle,inner sep=0pt] (p14) at (12, 0)
{\begin{tikzpicture}[scale=0.4]
\node at (0.9, 0.8) {\tiny $v_1$};
\end{tikzpicture}}
edge[-latex, loop, out=120, in=60, min distance=60, looseness=2] (p14)
edge[-latex, loop, out=140, in=40, min distance=120, looseness=2.2] (p14);
\node[circle,inner sep=0pt] (p15) at (16, 0)
{\begin{tikzpicture}[scale=0.4]
\node at (0.9, 0.8) {\tiny $v_2$};
\end{tikzpicture}}
edge[-latex, loop, out=120, in=60, min distance=60, looseness=2] (p15)
edge[-latex, loop, out=140, in=40, min distance=120, looseness=2.2] (p15);
\node[circle,inner sep=0pt] (p16) at (20, 0)
{\begin{tikzpicture}[scale=0.4]
\node at (0.9, 0.8) {\tiny $v_3$};
\end{tikzpicture}}
edge[-latex, loop, out=120, in=60, min distance=60, looseness=2] (p16)
edge[-latex, loop, out=140, in=40, min distance=120, looseness=2.2] (p16);
\draw[style=semithick, -latex] (p12.west)--(p11.east);
\draw[style=semithick, -latex] (p13.west)--(p12.east);
\draw[style=semithick, -latex] (p14.west)--(p13.east);
\draw[style=semithick, -latex] (p15.west)--(p14.east);
\draw[style=semithick, -latex] (p16.west)--(p15.east);
\node at (20,3) {\tiny $e_3$};
\node at (16,3) {\tiny $e_2$};
\node at (12,3) {\tiny $e_1$};
\node at (8,3) {\tiny $e_0$};
\node at (4,3) {\tiny $e_{-1}$};
\node at (0,3) {\tiny $e_{-2}$};
\end{tikzpicture}
\]

\noindent  If we identify the vertices of $\Omega$ with $\mathbb{Z}$, it is straightforward to see that every saturated hereditary subset of $\Omega^0$
is of the form $[n+1,+\infty)$ for some $n \in \mathbb{Z}$, or the empty set $\emptyset$, or $\Omega^0$. Similarly every maximal tail is
of the form $ \chi_n := ( - \infty , n ]$  for some $n \in \mathbb{Z}$, or $\Omega^0$ (recall that a maximal tail is nonempty). Following \cite[Example 5.9]{KaP} we can show that the primitive ideal space is homeomorphic to 
$\mathbb{Z} \cup \{ \infty \}$,  with the right-order topology.

By Theorem~\ref{thm:pure-inf} we may see that $C^* ( \Omega )$ is purely infinite since every vertex $v_m$ in the maximal tail $( - \infty , n ]$ is the range and source of the generalised cycle $( e_m , v_m )$ with an entrance.
\end{enumerate}
\end{examples}

\section{Real rank zero for higher rank $C^*$-algebras}\label{sec:real-rank}

\noindent
We now turn our attention to real rank zero. 
The following is a generalisation of \cite[Theorem 4.6]{JPS}.

\begin{theorem}\label{thm:RR0-finite-H}
Let $\Lambda$ be a locally convex, row-finite $k$-graph with no sources.
Suppose:
\begin{itemize}
  \item $\Lambda$ is strongly aperiodic;
  \item every vertex in each maximal tail $M$ is connected to by a
        generalised cycle in $M$;
  \item the lattice $H(\Lambda)$ of saturated hereditary subsets of
        $\Lambda^0$ is finite.
\end{itemize}
Then $C^*(\Lambda)$ is purely infinite and of topological dimension zero,
and the following are equivalent:
\begin{enumerate}
  \item $C^*(\Lambda)$ has real rank zero;
  \item $C^*(\Lambda)$ is $K_0$-liftable; that is, for every saturated
        hereditary $H\subseteq\Lambda^0$ the inclusion
        $I_H\hookrightarrow C^*(\Lambda)$ induces an injection
        $K_1(I_H)\to K_1(C^*(\Lambda))$.
\end{enumerate}
When $k=2$, condition~(2) is equivalent, by
\cite[Lemma~4.7 and Proposition~4.4]{PSS2}, to injectivity of the map
$H_1(j^H)$ on the homology of the Evans complex for every saturated
hereditary $H\subseteq\Lambda^0$, an elementary condition on the
connectivity matrices of $\Lambda$.
\end{theorem}

\begin{proof}
By the first two hypotheses and Theorem~\ref{thm:pure-inf}, $C^*(\Lambda)$ is
purely infinite. Since $H(\Lambda)$ is finite, the ideals of $C^*(\Lambda)$
are exactly the $I_H$ for $H\in H(\Lambda)$ (Theorems~\ref{thm:gauge-ideals}
and~\ref{thm:pure-inf}), so $\mathrm{Prim}(C^*(\Lambda))$ is a finite
$T_0$-space, in which every open set is compact; hence $C^*(\Lambda)$ has
topological dimension zero. Moreover $C^*(\Lambda)$ is separable
(Remark~\ref{rem:dense-span}) and nuclear \cite{KP2,S}.

A purely infinite $C^*$-algebra of topological dimension zero has real rank
zero if and only if it is $K_0$-liftable, by Pasnicu--R{\o}rdam
\cite[Theorem~4.2]{PR}. This gives the equivalence of (1) and~(2). The final
statement, for $k=2$, is \cite[Lemma~4.7 and Proposition~4.4]{PSS2}.
\end{proof}

\begin{remark}\label{rem:PSS2-6.1}
The $K_0$-liftability condition~(2) cannot be omitted: the first three
hypotheses of Theorem~\ref{thm:RR0-finite-H} do not by themselves imply real
rank zero. For each $n\ge 2$, Example~6.1 of \cite{PSS2} exhibits a finite,
row-finite, locally convex, strongly aperiodic $2$-graph $\Lambda$ whose
$C^*$-algebra is strongly purely infinite---so that every vertex of every
maximal tail is connected to by a generalised cycle
(Theorem~\ref{thm:pure-inf}), and $H(\Lambda)$ is finite---yet for which
$H_1(j^{H})$ fails to be injective at the saturated hereditary set
$H=\{w\}$, and consequently $C^*(\Lambda)$ does \emph{not} have real rank
zero \cite[Theorem~4.3]{PSS2}. The Examples~6.2 and~6.3 of \cite{PSS2} are
consistent with Theorem~\ref{thm:RR0-finite-H}: the former is not strongly
aperiodic, and in the latter (a simple stably finite algebra) not every
vertex is connected to by a generalised cycle, so that
Theorem~\ref{thm:pure-inf} excludes each from the hypotheses above.
\end{remark}

\noindent
The following definition comes from \cite[Definition 5.1]{A}.

\begin{definition}[Saturated morphism]\label{def:saturated-morphism}
Let $\Lambda,\Gamma$ be locally convex, row-finite $k$-graphs. A $k$-graph
morphism $\phi:\Gamma\to\Lambda$ (a degree-preserving functor) is a
\emph{saturated morphism} if it is injective and its image $\phi(\Gamma^0)$
is a saturated subset of $\Lambda^0$; that is, whenever $v\in\Lambda^0$ and
$n\in\N^k$ satisfy $s(v\Lambda^n)\subset\phi(\Gamma^0)$, then
$v\in\phi(\Gamma^0)$.
\end{definition}

\begin{lemma}\label{lem:saturated-morphism}
Let $\Lambda$ be a locally convex, row-finite $k$-graph with no sources, and
let $H\subset\Lambda^0$ be a saturated hereditary subset. Then the inclusion

\[
i : \Lambda H \hookrightarrow \Lambda
\]

is a saturated morphism in the sense of Definition~\ref{def:saturated-morphism}.
\end{lemma}

\begin{proof}
Recall that $\Lambda H$ is the full subgraph of $\Lambda$ with vertex set $H$
and morphisms

\[
(\Lambda H)^1 = \{\lambda\in\Lambda : s(\lambda)\in H\}.
\]

\noindent
Since $H$ is hereditary, $r(\lambda)\in H$ whenever $s(\lambda)\in H$, so
$\Lambda H$ is indeed a $k$-graph.

Let $v\in H$ and suppose that for some $n\in\mathbb{N}^k$ we have

\[
s(v\Lambda^n)\subset H.
\]

Because $H$ is saturated, this implies $v\in H$, which is already true.
Thus the saturation condition is automatically satisfied for all
vertices of $\Lambda H$.

Now let $\lambda\in\Lambda$ with $s(\lambda)\in H$. Then $\lambda\in\Lambda H$
by definition, so $i$ is surjective on morphisms with source in $H$.
Conversely, if $s(\lambda)\notin H$, then $\lambda$ does not lie in
$\Lambda H$, and the saturation condition imposes no requirement on $i$
at such vertices.

Therefore $i:\Lambda H\to\Lambda$ satisfies the saturated lifting
condition at every vertex of $\Lambda H$, and hence is a saturated
morphism.
\end{proof}

\begin{theorem}\label{thm:essential-subgraph}
Let $\Lambda$ be a locally convex, row-finite $k$-graph with no sources.
Define inductively

\[
V_0 := \{ v\in\Lambda^0 : v\Lambda^{\varepsilon_i} = \emptyset \text{ for some } i\},
\]

and for $n\geq 0$,

\[
V_{n+1} := V_n \cup \{ v\in\Lambda^0 : v\Lambda^{\varepsilon_i} \subset V_n\Lambda \text{ for some } i\}.
\]

Set $V := \bigcup_{n\geq 0} V_n$ and let

\[
E(\Lambda)^0 := \Lambda^0\setminus V,\qquad
E(\Lambda)^1 := \{ \lambda\in\Lambda : s(\lambda)\in E(\Lambda)^0\}.
\]

Then:
\begin{enumerate}
\item $E(\Lambda)$ is a locally convex, row-finite $k$-graph with no sources;
\item $E(\Lambda)$ is maximal with respect to having no sources: if $\Gamma\subset\Lambda$
      is a sub-$k$-graph with no sources, then $\Gamma\subset E(\Lambda)$;
\item $E(\Lambda)$ is unique with these properties.
\end{enumerate}
\end{theorem}

\begin{proof}
By construction, every vertex in $E(\Lambda)^0$ has at least one outgoing edge in each
coordinate, so $E(\Lambda)$ has no sources. Local convexity and row-finiteness are inherited
from $\Lambda$.

If $\Gamma\subset\Lambda$ is a sub-$k$-graph with no sources, then no vertex of $\Gamma^0$
can lie in $V_0$ (since vertices in $V_0$ are sources in some colour), and inductively no
vertex of $\Gamma^0$ can lie in any $V_n$. Thus $\Gamma^0\subset E(\Lambda)^0$, and
$\Gamma\subset E(\Lambda)$.

Uniqueness follows from maximality: any two maximal source-free subgraphs must coincide.
\end{proof}

\begin{theorem}\label{thm:quotient-sources}
Let $\Lambda$ be a locally convex, row-finite $k$-graph with no sources, and let
$V\subset\Lambda^0$ be the union of the $V_n$ as in Theorem~\ref{thm:essential-subgraph}.
Let $I_V\subset C^*(\Lambda)$ be the ideal generated by $\{s_v : v\in V\}$. Then

\[
C^*(\Lambda)/I_V \;\cong\; C^*(E(\Lambda)).
\]

\end{theorem}

\begin{proof}
By Theorem~\ref{thm:gauge-ideals}, the quotient by the gauge-invariant ideal generated
by vertices in $V$ is canonically isomorphic to the $k$-graph C$^*$-algebra of the
quotient graph obtained by removing $V$ and all edges with source in $V$. This is
exactly $E(\Lambda)$ by Theorem~\ref{thm:essential-subgraph}.
\end{proof}

\section{Extremal richness for higher rank $C^*$-algebras}\label{sec:extreme}

Let $A$ be a unital $C^*$-algebra and $\Ee (A)$, or just $\Ee$, the set of extreme points in the convex closed unit ball $A_1$ of $A$. By \cite[Proposition 1.4.7]{Ped} an extreme point of $A_1$ is characterised as a partial isometry $v$ satisfying $( 1 - v v^* ) A ( 1 - v^* v ) = 0$ (it is then said that the \emph{defect projections} $( 1 - v v^* ) $ and $( 1 - v^* v )$ are centrally orthogonal). We call elements $e \in \Ee A_+^{-1}$ ($=A^{-1}\Ee A^{-1}$) \emph{quasi-invertible} (see \cite{BP,Pe0,BP2}). We denote the set of all quasi invertible elements of $A$ by $A_q^{-1}$.

\begin{definition}[Extremely rich]
Let $A$ be a unital $C^*$-algebra, then $A$ is said to be \emph{extremely rich} if $A_q^{-1}$ is dense in $A$. For a non-unital $C^*$-algebra $A$, $A$ is said to be extremely rich if its unitisation $\widetilde{A}$ is so.
\end{definition}

\noindent
\textbf{Basic properties:}
\begin{itemize}
\item[(i)] A $C^*$ algebra $A$ has stable rank one if $A^{-1}$ the set of invertible elements is dense in $A$, since $A^{-1} \subset A^{-1}_q$ it follows that a $C^*$-algebra with stable rank one is extremely rich.
\item[(ii)] By \cite[Theorem 10.1]{Pe0} purely infinite simple $C^*$-algebras are also extremely rich (note Pedersen uses the Cuntz definition of purely infinite).
\item[(iii)] By \cite[Corollary 2.9]{BP3} if $A$ is an extremely rich simple $C^*$-algebra then it is either purely infinite or it has stable rank one.
\item[(iv)] By combining several of the results in \cite{BP} we can see that all AF algebras are extremely rich.
\end{itemize}

\begin{theorem}\label{thm:extreme-rich-structure}
Let $\Lambda$ be a locally convex, row-finite $k$-graph with no sources.
Assume:
\begin{enumerate}
\item $C^*(\Lambda)$ has topological dimension zero;
\item every ideal of $C^*(\Lambda)$ is gauge-invariant;
\item for each saturated hereditary $H\subset\Lambda^0$, the quotient
      $C^*(\Gamma(\Lambda\setminus H))$ is either AF or purely infinite.
\end{enumerate}
Then $C^*(\Lambda)$ is extremely rich.
\end{theorem}

\begin{proof}
By (1) and (2), the ideal lattice of $C^*(\Lambda)$ is controlled by saturated
hereditary subsets and all ideals are gauge-invariant (Theorem~\ref{thm:gauge-ideals}).
By (3), every quotient by such an ideal is either AF or purely infinite. AF algebras
are extremely rich, and purely infinite C$^*$-algebras are extremely rich by
\cite{KR,Pe0}. Thus every quotient of $C^*(\Lambda)$ is extremely rich.

Brown--Pedersen’s permanence results \cite{BP2,BP3} imply that if all quotients
by primitive ideals are extremely rich, then $C^*(\Lambda)$ itself is extremely rich.
\end{proof}

\begin{theorem}\label{thm:RR0-extreme}
Let $\Lambda$ be a locally convex, row-finite, strongly aperiodic $k$-graph with no sources.
Assume:
\begin{enumerate}
\item every vertex in each maximal tail is connected to by a generalised cycle;
\item the lattice $H(\Lambda)$ of saturated hereditary subsets is finite;
\item for each saturated hereditary $H\subset\Lambda^0$, the quotient
      $C^*(\Gamma(\Lambda\setminus H))$ is either AF or purely infinite;
\item $C^*(\Lambda)$ is $K_0$-liftable
      (Theorem~\ref{thm:RR0-finite-H}(2)).
\end{enumerate}
Then $C^*(\Lambda)$ has real rank zero and is extremely rich.
\end{theorem}

\begin{proof}
By hypotheses (1) and (2) and Theorem~\ref{thm:pure-inf}, $C^*(\Lambda)$ is
purely infinite, so real rank zero follows from
Theorem~\ref{thm:RR0-finite-H} together with hypothesis (4). Extremal richness
follows from Theorem~\ref{thm:extreme-rich-structure}.
\end{proof}

\begin{theorem}\label{thm:approx}
Let $\Lambda$ be a locally convex, row-finite $k$-graph with no sources.
Assume:
\begin{enumerate}
\item $\Lambda$ is strongly aperiodic;
\item every vertex in each maximal tail $M$ is connected to by a
      generalised cycle in $M$;
\item the lattice $H(\Lambda)$ of saturated hereditary subsets of
      $\Lambda^0$ is finite.
\end{enumerate}
Then there exists an increasing sequence of $C^*$-subalgebras

\[
A_1 \subset A_2 \subset \cdots \subset A_n = C^*(\Lambda)
\]

and a sequence of locally convex, row-finite, strongly aperiodic
$k$-graphs

\[
\Gamma_1 \subset \Gamma_2 \subset \cdots \subset \Gamma_n = \Lambda
\]

such that each $A_i \cong C^*(\Gamma_i)$ and

\[
\overline{\bigcup_{i=1}^n A_i} = C^*(\Lambda).
\]

\end{theorem}

\begin{proof}
Since $H(\Lambda)$ is finite, we may list its elements in increasing
order:

\[
\emptyset = H_0 \subsetneq H_1 \subsetneq \cdots \subsetneq H_m = \Lambda^0.
\]

For each $i$ let

\[
\Gamma_i := \Lambda\setminus H_{m-i},
\]

the full subgraph on the vertex set $\Lambda^0\setminus H_{m-i}$ with
all morphisms whose source lies in $\Lambda^0\setminus H_{m-i}$.
Because $H_{m-i}$ is saturated hereditary, $\Gamma_i$ is a locally
convex, row-finite $k$-graph with no sources.

\smallskip

\textbf{Claim 1: Each $\Gamma_i$ is strongly aperiodic.}
Since $\Lambda$ is strongly aperiodic, every quotient
$\Gamma(\Lambda\setminus H)$ is aperiodic for every saturated hereditary
$H\subset\Lambda^0$. But $\Gamma_i = \Gamma(\Lambda\setminus H_{m-i})$,
so each $\Gamma_i$ is aperiodic. Moreover, if $K\subset\Gamma_i^0$
is saturated hereditary in $\Gamma_i$, then $K\cup H_{m-i}$ is
saturated hereditary in $\Lambda$, and strong aperiodicity of
$\Lambda$ implies aperiodicity of $\Gamma(\Lambda\setminus (K\cup H_{m-i}))$.
Thus $\Gamma_i$ is strongly aperiodic.

\smallskip

\textbf{Claim 2: Each vertex of each maximal tail of $\Gamma_i$ is
connected to by a generalised cycle in $\Gamma_i$.}
Let $M$ be a maximal tail in $\Gamma_i$. Then $M$ is also a maximal
tail in $\Lambda$ (because adding vertices in $H_{m-i}$ cannot create
new paths into $M$). By hypothesis, every vertex of $M$ is connected
to by a generalised cycle in $\Lambda$. Since $M\subset\Gamma_i^0$,
and generalised cycles based at vertices of $M$ remain inside
$\Gamma_i$, the claim follows.

\smallskip

\textbf{Define the subalgebras.}
For each $i$ let

\[
A_i := C^*(\Gamma_i) \subset C^*(\Lambda),
\]

via the canonical inclusion of the Cuntz--Krieger family
$\{s_\lambda : \lambda\in\Gamma_i\}$ into $C^*(\Lambda)$.

\smallskip

\textbf{Claim 3: $A_i \subset A_{i+1}$ and $A_m = C^*(\Lambda)$.}
Since $\Gamma_i^0 \subset \Gamma_{i+1}^0$ and
$\Gamma_i^1 \subset \Gamma_{i+1}^1$, the universal property of
$C^*(\Gamma_i)$ gives an inclusion $A_i\subset A_{i+1}$.  
For $i=m$ we have $\Gamma_m = \Lambda$, so $A_m=C^*(\Lambda)$.

\smallskip

\textbf{Claim 4: $\overline{\bigcup_i A_i} = C^*(\Lambda)$.}
Every generator $s_\lambda$ of $C^*(\Lambda)$ lies in some
$A_i$, namely the smallest $i$ such that $s(\lambda)\notin H_{m-i}$.
Thus the Cuntz--Krieger family for $\Lambda$ is contained in
$\bigcup_i A_i$, and hence

\[
C^*(\Lambda) = \overline{\operatorname{span}}
\{ s_\lambda s_\mu^* : \lambda,\mu\in\Lambda\}
\subset \overline{\bigcup_i A_i}.
\]

\smallskip

Combining the claims proves the theorem.
\end{proof}

\subsection{Examples: AF, purely infinite, and mixed extremal richness}

\begin{example}[AF $k$-graph]\label{ex:AF-kgraph}
Fix $k\ge 1$. Let $\Omega_k$ be the $k$-graph, defined in
Example~\ref{ex:skew-tikz}, with
vertex set $\mathbb N^k$ and degree-$\varepsilon_i$
edges
\[
(n_1,\dots,n_k) \xleftarrow{e^{(i)}_{n_i}} (n_1,\dots,n_i+1,\dots,n_k),
\qquad i=1,\dots,k,
\]

\noindent
with factorisation rules given by coordinatewise commutation
$e^{(i)}_{n_i} e^{(j)}_{n_j} = e^{(j)}_{n_j} e^{(i)}_{n_i}$ for $i\neq j$.
Then $\Omega_k$ is row-finite, locally convex, has no
sources, and $C^*(\Omega_k) \cong \mathcal{K} ( \ell^2 ( \mathbb{N}^2 ))$ is AF, hence real rank zero
and extremely rich.
\end{example}


\par\null
\begin{minipage}{0.75 \linewidth}
\begin{example}[Purely infinite rank-2 graph]\label{ex:PI-kgraph}
Let $\Lambda_{\mathrm{PI}}$ be the rank-2 graph with a single vertex
$v$, degree-$\varepsilon_1$ edges $f_1,f_2$ and degree-$\varepsilon_2$ edges $g_1,g_2$,
with factorisation rules
\[
g_j f_1 = f_2 g_j,\qquad g_j f_2 = f_1 g_j,\qquad j=1,2.
\]

Then $\Lambda_{\mathrm{PI}}$ is row-finite, locally convex and cofinal.
Regarded as a single-vertex $2$-graph $\mathbb F_\theta^+$ in the sense
of Davidson--Yang \cite{DY}, with the degree-$\varepsilon_2$ edges
$g_1,g_2$ playing the role of the $e$-generators and the
degree-$\varepsilon_1$ edges $f_1,f_2$ that of the $f$-generators, the
commutation rules give $\theta(i,j)=(i,\bar\jmath)$, where
$\bar 1=2$, $\bar 2=1$; in the factorisation $\theta(i,j)=(\beta_j(i),\alpha_i(j))$
this means $\alpha_1=\alpha_2=(1\,2)$ and $\beta_1=\beta_2=\mathrm{id}$.
Taking $B=\{1,2\}$ and the length-one word $i_1=1$, we have
$\alpha_{1}(B)=B$ with $|B|=2$, so $\Lambda_{\mathrm{PI}}$ is aperiodic
by \cite[Corollary~4.1]{DY}. Being cofinal, it is therefore strongly
aperiodic. Every vertex is connected to by a generalised
cycle with an entrance, so by Theorem~\ref{thm:pure-inf}
$C^*(\Lambda_{\mathrm{PI}})$ is purely infinite, real rank zero, and
extremely rich.
\end{example}
\end{minipage}%
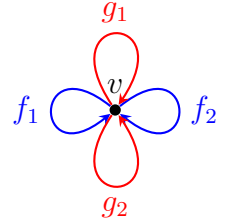
\begin{minipage}[t!]{0.25 \linewidth}
\[
\begin{tikzpicture}[scale=1.1,>=stealth]
  \node (v) at (0,0) [circle,fill=black,inner sep=1.5pt,label=above:{$v$}] {};
  \draw[->,blue,thick] (v) .. controls (-1,0.8) and (-1,-0.8) .. (v)
    node[midway,left] {$f_1$};
  \draw[->,blue,thick] (v) .. controls (1,0.8) and (1,-0.8) .. (v)
    node[midway,right] {$f_2$};
  \draw[->,red,thick] (v) .. controls (-0.8,1.2) and (0.8,1.2) .. (v)
    node[midway,above] {$g_1$};
  \draw[->,red,thick] (v) .. controls (-0.8,-1.2) and (0.8,-1.2) .. (v)
    node[midway,below] {$g_2$};
\end{tikzpicture}
\]
\end{minipage}

\bigskip

\begin{minipage}{0.6 \linewidth}
\begin{example}[Mixed: AF ideal, purely infinite quotient]\label{ex:mixed-kgraph}
Let $\Lambda_{\mathrm{mix}}$ be obtained from $\Lambda_{\mathrm{PI}}$
by adjoining a new vertex $w$ together with a degree-$\varepsilon_1$ edge
$\alpha:w\to v$ and a degree-$\varepsilon_2$ edge $\beta:w\to v$, subject to
the factorisation rules $g_i\alpha=f_i\beta$ for $i=1,2$. (A single edge from
$w$ would not give a $k$-graph: since $v$ emits the red loops $g_1,g_2$, the
degree-$(\varepsilon_1+\varepsilon_2)$ paths $g_i\alpha$ would have no
factorisation as a red-then-blue path unless $w$ also emits a red edge into
$v$.) Then $\Lambda_{\mathrm{mix}}$ is a row-finite, locally convex $2$-graph
with no sources. The hereditary subset $H=\{w\}$ generates an AF ideal
$I_H\cong K(\ell^2(\Lambda_{\mathrm{mix}}w))$, while the quotient
$C^*(\Lambda_{\mathrm{mix}})/I_H\cong C^*(\Lambda_{\mathrm{PI}})$ is
purely infinite. Since $w$ is not connected to by any generalised cycle,
$\Lambda_{\mathrm{mix}}$ lies outside the hypotheses of
Theorem~\ref{thm:RR0-finite-H} (indeed $C^*(\Lambda_{\mathrm{mix}})$ is not
purely infinite). Real rank zero instead follows from the Brown--Pedersen
extension theorem \cite[Theorem~3.14]{BP}: the ideal $I_H$ is AF, so
$\mathrm{RR}(I_H)=0$ and $K_1(I_H)=0$; the quotient
$C^*(\Lambda_{\mathrm{PI}})$ is simple and purely infinite, hence of real
rank zero; and the vanishing of $K_1(I_H)$ guarantees that projections in the
quotient lift, so $\mathrm{RR}(C^*(\Lambda_{\mathrm{mix}}))=0$. Extreme
richness follows from Theorem~\ref{thm:extreme-rich-structure}, since every
quotient of $C^*(\Lambda_{\mathrm{mix}})$ is AF or purely infinite.
\end{example}
\end{minipage}%
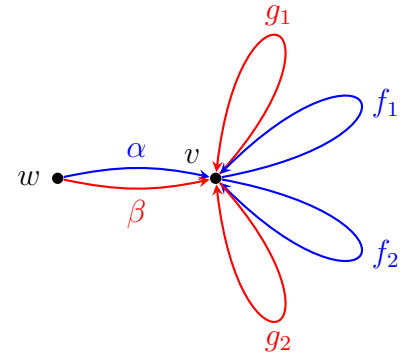
\begin{minipage}[t!]{0.4 \linewidth}
\[
\begin{tikzpicture}[scale=1.1,>=stealth]
  \node (w) at (-1.9,0) [circle,fill=black,inner sep=1.5pt,label=left:{$w$}] {};
  \node (v) at (0,0) [circle,fill=black,inner sep=1.5pt,label=above left:{$v$}] {};
  \draw[->,blue,thick] (w) to[bend left=12]
    node[midway,above] {$\alpha$} (v);
  \draw[->,red,thick] (w) to[bend right=12]
    node[midway,below] {$\beta$} (v);
  \draw[->,red,thick]  (v) .. controls (1.7,1.9) and (0.3,2.6) .. (v)
    node[midway,above] {$g_1$};
  \draw[->,blue,thick] (v) .. controls (2.7,0.5) and (1.9,1.9) .. (v)
    node[midway,right] {$f_1$};
  \draw[->,blue,thick] (v) .. controls (2.7,-0.5) and (1.9,-1.9) .. (v)
    node[midway,right] {$f_2$};
  \draw[->,red,thick]  (v) .. controls (1.7,-1.9) and (0.3,-2.6) .. (v)
    node[midway,below] {$g_2$};
\end{tikzpicture}
\]
\end{minipage}

\section{Chain conditions, real rank, and extremal richness}\label{sec:chain}

In this section we record several structural consequences of the
aperiodicity and ideal–lattice hypotheses developed earlier. Throughout,
$\Lambda$ denotes a locally convex, row-finite $k$-graph with no sources.

\subsection{Chain conditions}

\begin{definition}
A $C^*$-algebra $A$ is called \emph{Noetherian} if it satisfies the ascending chain condition for closed ideals, that is, for any ascending chain $I_1 \subseteq I_2 \subseteq I_3 \subseteq \cdots$ of closed ideals of $A$, there is a positive integer $n$ such that $I_i = I_n$, for all $i \geq n$. The dual notion to a Noetherian $C^*$-algebra is that of an \emph{Artinian} $C^*$-algebra, which satisfies the descending chain condition for closed ideals. 
\end{definition}

\begin{theorem}\label{thm:chain-conditions}
If $C^*(\Lambda)$ has topological dimension zero, then $C^*(\Lambda)$
satisfies both the ascending and descending chain conditions on ideals.
\end{theorem}

\begin{proof}
Topological dimension zero implies that $\Prim(C^*(\Lambda))$ has a basis
of compact-open sets. By Pedersen's ideal–lattice theory
\cite[Theorem~4.4.6]{Ped}, separable $C^*$-algebras with totally
disconnected primitive ideal space satisfy both ACC and DCC on ideals.
\end{proof}

\subsection{Stable rank}

\begin{theorem}\label{thm:stable-rank}
Let $A=C^*(\Lambda)$. Suppose:
\begin{enumerate}
\item $RR(A)=0$;
\item every primitive quotient of $A$ is AF.
\end{enumerate}
Then $A$ has stable rank one.
\end{theorem}

\begin{proof}
By (2) every primitive quotient of $A$ is AF, hence extremely rich
\cite{BP} and of stable rank one; in particular no primitive quotient of
$A$ is purely infinite. Thus $A$ is extremely rich of real rank zero with
no purely infinite primitive quotient, and so falls on the stable-rank-one
branch of the Brown--Pedersen dichotomy \cite[Corollary~2.9]{BP3}. Hence
$\operatorname{sr}(A)=1$.
\end{proof}

\begin{remark}
The hypothesis that every primitive quotient be AF cannot be relaxed to
allow purely infinite quotients. A purely infinite simple $C^*$-algebra is
extremely rich \cite{Pe0} and of real rank zero, but has stable rank
infinity; by the dichotomy \cite[Corollary~2.9]{BP3} it lies on the purely
infinite branch, not the stable-rank-one branch. Consequently the
mixed algebras of Theorem~\ref{thm:mixed-extreme}, whose maximal-tail
hypothesis forces purely infinite quotients, are \emph{not} of stable rank
one.
\end{remark}

\subsection{Mixed AF/purely infinite behaviour}

\begin{theorem}\label{thm:mixed-extreme}
Let $\Lambda$ be strongly aperiodic, row-finite, locally convex, and
assume:
\begin{enumerate}
\item every vertex in each maximal tail is connected to by a generalised
      cycle;
\item the lattice $H(\Lambda)$ of saturated hereditary subsets is finite;
\item for each saturated hereditary $H\subset\Lambda^0$, the quotient
      $C^*(\Gamma(\Lambda\setminus H))$ is either AF or purely infinite;
\item $C^*(\Lambda)$ is $K_0$-liftable (Theorem~\ref{thm:RR0-finite-H}(2)).
\end{enumerate}
Then $C^*(\Lambda)$ is purely infinite, has real rank zero, and is extremely
rich. It is not stably finite, and hence does not have stable rank one.
\end{theorem}

\begin{proof}
By hypothesis (1) and Theorem~\ref{thm:pure-inf}, $C^*(\Lambda)$ is purely
infinite; in particular every quotient in hypothesis (3) is purely infinite.
Real rank zero then follows from Theorem~\ref{thm:RR0-finite-H} together with
hypothesis (4), and extremal richness from
Theorem~\ref{thm:extreme-rich-structure}. A purely infinite $C^*$-algebra
contains an infinite projection, so it is not stably finite and therefore not
of stable rank one.
\end{proof}

\begin{remark}
The hypotheses above are satisfied for many natural classes of
higher-rank graphs, including those with aperiodic quartets at every
vertex, finite vertex set, cofinality, and $K_0$-liftable
$C^*$-algebra. Note that the mixed example
$\Lambda_{\mathrm{mix}}$ of Example~\ref{ex:mixed-kgraph} does \emph{not}
satisfy hypothesis (1): the vertex $w$ is not connected to by any
generalised cycle, and correspondingly $C^*(\Lambda_{\mathrm{mix}})$ is not
purely infinite. Its real rank zero is established separately in
Example~\ref{ex:mixed-kgraph}.
\end{remark}

\begin{examples}\label{ex:5.5}
We conclude this section with three illustrative $k$-graph examples
demonstrating the AF, purely infinite, and mixed behaviours that arise
under the hypotheses of Theorems~\ref{thm:RR0-finite-H},
\ref{thm:extreme-rich-structure}, and \ref{thm:RR0-extreme}.

\smallskip

\noindent\textbf{(1) An AF $k$-graph.}
Let $\Lambda_{\mathrm{AF}}$ be the product-of-chains $k$-graph described in
Example~\ref{ex:AF-kgraph}. Since $\Lambda_{\mathrm{AF}}$ has no cycles or
generalised cycles, $C^*(\Lambda_{\mathrm{AF}})$ is AF. In particular,
$RR(C^*(\Lambda_{\mathrm{AF}}))=0$, $K_1(C^*(\Lambda_{\mathrm{AF}}))=0$, and
$C^*(\Lambda_{\mathrm{AF}})$ is extremely rich.

\smallskip

\noindent\textbf{(2) A purely infinite rank-$2$ graph.}
Let $\Lambda_{\mathrm{PI}}$ be the one-vertex rank-$2$ graph of
Example~\ref{ex:PI-kgraph}. As computed there, $\alpha_1=\alpha_2=(1\,2)$
fixes $B=\{1,2\}$, so $\Lambda_{\mathrm{PI}}$ is aperiodic by
\cite[Corollary~4.1]{DY} and, being cofinal, strongly aperiodic. Every
vertex is connected to by a generalised cycle with an entrance. By Theorem~\ref{thm:pure-inf},
$C^*(\Lambda_{\mathrm{PI}})$ is purely infinite. Since $H(\Lambda_{\mathrm{PI}})$
is finite, Theorem~\ref{thm:RR0-finite-H} gives $RR(C^*(\Lambda_{\mathrm{PI}}))=0$,
and hence $C^*(\Lambda_{\mathrm{PI}})$ is extremely rich.

\smallskip

\noindent\textbf{(3) A mixed example: AF ideal and purely infinite quotient.}
Let $\Lambda_{\mathrm{mix}}$ be the $2$-graph obtained by adjoining a new
vertex $w$ together with a degree-$\varepsilon_1$ edge $\alpha:w\to v$ and a
degree-$\varepsilon_2$ edge $\beta:w\to v$ to $\Lambda_{\mathrm{PI}}$
(with $g_i\alpha=f_i\beta$), as in Example~\ref{ex:mixed-kgraph}. The
hereditary subset $H=\{w\}$ generates an AF ideal
$I_H\cong K(\ell^2(\Lambda_{\mathrm{mix}}w))$, while the quotient
$C^*(\Lambda_{\mathrm{mix}})/I_H\cong C^*(\Lambda_{\mathrm{PI}})$ is purely
infinite. The lattice $H(\Lambda_{\mathrm{mix}})$ is finite, and
$\Lambda_{\mathrm{mix}}$ is strongly aperiodic with generalised cycles at $v$.
The vertex $w$ is not connected to by a generalised cycle, so
$\Lambda_{\mathrm{mix}}$ does not satisfy the hypotheses of
Theorem~\ref{thm:RR0-finite-H}; as in Example~\ref{ex:mixed-kgraph}, real rank
zero instead follows from the Brown--Pedersen extension theorem
\cite[Theorem~3.14]{BP}, since the AF ideal $I_H$ satisfies
$\mathrm{RR}(I_H)=0$ and $K_1(I_H)=0$ and the quotient
$C^*(\Lambda_{\mathrm{PI}})$ has real rank zero. Then
Theorem~\ref{thm:extreme-rich-structure} implies that
$C^*(\Lambda_{\mathrm{mix}})$ is extremely rich. Note that
$C^*(\Lambda_{\mathrm{mix}})$ does \emph{not} have stable rank one: its
quotient $C^*(\Lambda_{\mathrm{PI}})$ is purely infinite, so
$C^*(\Lambda_{\mathrm{mix}})$ is not stably finite and lies on the purely
infinite branch of the dichotomy \cite[Corollary~2.9]{BP3}.
\end{examples}

\section{Trichotomy, $\mathcal{Z}$-stability}\label{sec:trichotomy}

By \cite[Theorem~2.3]{PRA2}, if $A$ is a Noetherian, purely infinite, separable
$C^*$-algebra, then $A$ is generated (as a closed ideal) by a single projection.
Since every nonzero projection in a purely infinite algebra is properly infinite,
such an algebra cannot be stably finite. We record the consequence as follows.

\begin{theorem}[Trichotomy]\label{thm:trichotomy}
There is no Noetherian, separable $C^*$-algebra that is both purely infinite
and stably finite. In particular, there is no Noetherian purely infinite
$AH_0$-algebra.
\end{theorem}

Thus every Noetherian separable $C^*$-algebra belongs to exactly one of the
following three classes:

\[
\text{(i) purely infinite},\qquad
\text{(ii) neither purely infinite nor stably finite},\qquad
\text{(iii) stably finite}.
\]

\begin{remark}
There exist $C^*$-algebras that are both purely infinite and stably finite
(e.g.\ the minimal unitisation of $\mathcal O_2\otimes\mathcal K$), but none of
these are Noetherian.
\end{remark}

We now show that higher-rank graph $C^*$-algebras occur in each branch of the
trichotomy.

\subsection{Purely infinite, not stably finite}

\begin{itemize}
\item
\textbf{Non-simple case.}
Let $\Lambda$ be a finite, row-finite, locally convex, strongly aperiodic
$k$-graph in which every vertex in each maximal tail is connected to by a
generalised cycle with an entrance. Then $C^*(\Lambda)$ is purely infinite by
Theorem~\ref{thm:pure-inf}, and not stably finite since it contains infinite
projections \cite[Proposition~2.7]{PSS4}.

\item
\textbf{Simple case.}
Purely infinite simple higher-rank graph algebras appear in
\cite[\S7.3]{CaHS}.

\item
\textbf{Twisted case.}
By \cite[Proposition~5.7]{KPS5}, if $\Lambda$ is row-finite with no sources,
aperiodic, and every vertex can be reached from a generalised cycle with an
entrance, then for every $2$-cocycle $c\in Z^2(\Lambda,\mathbb T)$, each
hereditary subalgebra of $C^*(\Lambda,c)$ contains an infinite projection.
If $\Lambda$ is cofinal, then $C^*(\Lambda,c)$ is simple and purely infinite.

\item
\textbf{Amplified directed graphs.}
Let $\Lambda$ be an amplified directed graph (i.e.\ $v\Lambda w$ is either
empty or infinite for all $v,w$). Then every vertex set is saturated and
$\Lambda$ satisfies condition (K), hence is strongly aperiodic
\cite[Lemma~3.2]{KaP}. If $\Lambda^1\neq\emptyset$, then $\Lambda$ is strongly
connected and $C^*(\Lambda)$ is purely infinite and simple
\cite[\S6.1]{ERS}. If $\Lambda^0$ is finite, then $\Prim(C^*(\Lambda))$ is
finite, so $C^*(\Lambda)$ is Noetherian (and Artinian).
\end{itemize}

\subsection{Neither purely infinite nor stably finite}

\begin{itemize}
\item Consider the diagram of the directed graph $\Sigma_2$ shown below. The vertex $w$ carries two loops, so
$p_w$ is infinite; the sink $v$ lies on no cycle, and $\{v\}$ is a maximal tail
that is not connected to by a cycle with an entrance.
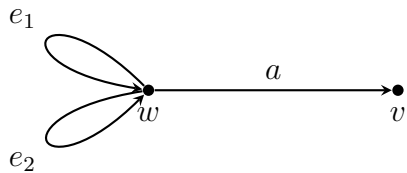
\[
\begin{tikzpicture}[scale=1.1,>=stealth]
  \node (w) at (0,0) [circle,fill=black,inner sep=1.5pt,label=below:{$w$}] {};
  \node (v) at (3,0) [circle,fill=black,inner sep=1.5pt,label=below:{$v$}] {};
  \draw[->,thick] (w) .. controls (-1.3,1.3) and (-1.9,0.3) .. (w)
    node[midway,above left] {$e_1$};
  \draw[->,thick] (w) .. controls (-1.9,-0.3) and (-1.3,-1.3) .. (w)
    node[midway,below left] {$e_2$};
  \draw[->,thick] (w) -- (v) node[midway,above] {$a$};
\end{tikzpicture}
\]

\item
Let $\Sigma_2$ be the directed graph shown above. The
directed graph underlying $\Sigma_2$ satisfies condition (K), hence
$\Sigma_2$ is strongly aperiodic. The set $\{v\}\subset\Sigma_2^0$ is a
maximal tail, but it is not connected to by a cycle with an entrance in that
tail. By Theorem~\ref{thm:pure-inf}, $C^*(\Sigma_2)$ is not purely infinite.
Moreover, $p_w$ is an infinite projection \cite[Proposition~2.7]{PSS4}, so
$C^*(\Sigma_2)$ is not stably finite. Since $\Sigma_2^0$ is finite,
$C^*(\Sigma_2)$ is Noetherian and Artinian.
\end{itemize}

\subsection{Stably finite, not purely infinite}

\begin{itemize}
\item
All AF algebras with finitely many ideals (up to Morita equivalence)
\cite{KPR,T,ES}.

\item
The $A\mathbb T$-algebras associated to rank-$2$ Bratteli diagrams
\cite{PRRS}, including the irrational rotation algebra and Bunce--Deddens
algebras \cite[Theorem~5.1, \S7.1]{CaHS}.

\item
If $\Lambda$ is a finite, row-finite, locally convex $2$-graph with no cycle
with an entrance and cofinal, then $C^*(\Lambda)$ is a unital $A\mathbb T$-algebra
\cite[Theorem~2.5]{PSS4}.

\item
AF-embeddable higher-rank graph algebras as in \cite[\S7.2--7.3]{CaHS}.

\item
Finite amplified $1$-graphs with no loops (hence no cycles) are stably finite
\cite{ERS,KPR,ES}.
\end{itemize}

\begin{theorem}\label{thm:D0-stable}
Let $A$ be a $C^*$\nobreakdash-algebra which is both Noetherian and Artinian. 
Suppose that $A$ is separable, nuclear and purely infinite, and let $D_0$ be either 
$\mathcal{O}_2$ or the Jiang--Su algebra $\mathcal{Z}$. Then $A$ is $D_0$-stable, 
that is, $A \cong A \otimes D_0$. Moreover, for every nonzero $C^*$\nobreakdash-algebra $B$, 
the tensor product $A \otimes B$ is purely infinite.
\end{theorem}

\begin{proof}
The $D_0$-stability follows from \cite[Theorem~2.4]{PRA1}. 
The pure infiniteness of $A \otimes B$ for nonzero $B$ follows from 
\cite[Proposition~4.7]{KR}.
\end{proof}

\begin{proposition}\label{prop:Oinfty-absorption}
Let $\Lambda$ be a locally convex, row-finite $k$-graph with no sources
satisfying condition~\textup{(\ref{pi:e})} of Theorem~\ref{thm:pure-inf}, so that
$C^*(\Lambda)$ is purely infinite in the sense of
Definition~\ref{def:pure-inf}. Suppose in addition that $C^*(\Lambda)$ has
topological dimension zero. Then $C^*(\Lambda)$ is strongly purely infinite,
\[
C^*(\Lambda) \;\cong\; C^*(\Lambda)\otimes\mathcal O_\infty,
\]
and $C^*(\Lambda)\otimes B$ is purely infinite for every nonzero
$C^*$-algebra $B$. In particular, the hypotheses hold whenever the lattice
$H(\Lambda)$ of saturated hereditary subsets is finite. Moreover
$C^*(\Lambda)$ has nuclear dimension one, even when $C^*(\Lambda)$ is not
simple (cf.\ \cite{RSS}).
\end{proposition}

\begin{proof}
\emph{Separability and nuclearity.} By Remark~\ref{rem:dense-span},
$C^*(\Lambda)$ is separable; it is nuclear as the $C^*$-algebra of the
amenable étale path groupoid $\mathcal G_\Lambda$ \cite{KP2,S}, and lies in
the bootstrap class.

\emph{Ideal structure and topological dimension zero.} By
Theorem~\ref{thm:gauge-invariant}, strong aperiodicity makes every ideal of
$C^*(\Lambda)$ gauge invariant, so $H\mapsto I_H$ is a lattice isomorphism of
$H(\Lambda)$ onto the ideal lattice of $C^*(\Lambda)$, with quotients
$C^*(\Gamma(\Lambda\setminus H))$ (Theorem~\ref{thm:gauge-ideals}). Hence
$\Prim(C^*(\Lambda))$ is the space of maximal tails, topologised by the
containment order on $H(\Lambda)$. When $H(\Lambda)$ is finite this is a
finite $T_0$-space, in which every open set is compact; its topology is then a
basis of compact-open sets, so $C^*(\Lambda)$ has topological dimension zero.

\emph{Pure infiniteness implies strong pure infiniteness.} A separable,
nuclear $C^*$-algebra of topological dimension zero is purely infinite if and
only if it is strongly purely infinite: the two notions coincide once
$\Prim$ has a basis of compact-open sets.
Since $C^*(\Lambda)$ is purely infinite by hypothesis, it is therefore
strongly purely infinite.

\emph{Absorption and permanence.} By the Kirchberg--R\o rdam
$\mathcal O_\infty$-absorption theorem \cite{KR}, a separable, nuclear,
strongly purely infinite $C^*$-algebra $D$ satisfies
$D\cong D\otimes\mathcal O_\infty$; applying this to $D=C^*(\Lambda)$ gives
$C^*(\Lambda)\cong C^*(\Lambda)\otimes\mathcal O_\infty$. Finally, for any
nonzero $B$,
\[
C^*(\Lambda)\otimes B \;\cong\; C^*(\Lambda)\otimes\mathcal O_\infty\otimes B
\]
is purely infinite by \cite[Proposition~4.7]{KR}, the permanence property
already used in the proof of Theorem~\ref{thm:D0-stable}. For the last part,
that $C^*(\Lambda)$ has nuclear dimension one, recall that $C^*(\Lambda)$ is
$\mathcal O_\infty$-stable and apply \cite[Theorem~A]{BGSW}; that strongly
purely infinite $C^*$-algebras have finite nuclear dimension is due to Szab\'o
\cite{Sz}.
\end{proof}

\begin{remark}
Proposition~\ref{prop:Oinfty-absorption} subsumes the purely infinite branch of
Corollary~\ref{cor:Z-stable}: $\mathcal O_\infty$-stability follows from pure
infiniteness together with topological dimension zero alone, with no
Noetherian or Artinian hypothesis and without passing through
Theorem~\ref{thm:D0-stable}. Since topological dimension zero is strictly
weaker than finiteness of $H(\Lambda)$---it holds whenever the maximal-tail
space is totally disconnected---the conclusion covers many $k$-graphs with
infinite ideal lattice.
\end{remark}

\begin{remark}\label{rem:mixed-nucdim}
Proposition~\ref{prop:Oinfty-absorption} computes the nuclear dimension only in
the purely infinite case. For the mixed algebras of
Example~\ref{ex:mixed-kgraph}, which have an AF ideal $I$ and a purely infinite
quotient of topological dimension zero, the Winter--Zacharias extension estimate
\cite{WZ} gives only
\[
\dim_{\mathrm{nuc}} C^*(\Lambda)
  \;\le\; \dim_{\mathrm{nuc}} I + \dim_{\mathrm{nuc}}\big(C^*(\Lambda)/I\big) + 1
  \;=\; 0 + 1 + 1 \;=\; 2 .
\]
In the opposite direction---a purely infinite ideal with AF quotient---Ruiz,
Sims and Tomforde \cite[Theorem~6.1]{RST} obtain the sharp value one, using that
such graph extensions are quasidiagonal. An extension with purely infinite
quotient is not quasidiagonal, so the quasidiagonality argument does not apply
in the present AF-ideal direction. For $k=1$, however, Faurot and Schafhauser
\cite[Theorem~B]{FS} nonetheless obtain the sharp value one for this
AF-ideal, purely-infinite-quotient shape, via Evington's full-extension theorem
\cite{Ev} in place of quasidiagonality: when the AF ideal $I_H$ is stable and
the extension is full---which for a finite $1$-graph amounts to every source
connecting to every cycle---the bound drops from two to one. The higher-rank
analogue requires a combinatorial guarantee of stability of $I_H$ and fullness
of the extension; the mixed graph $\Lambda_{\mathrm{mix}}$ of
Example~\ref{ex:mixed-kgraph}, whose vertex $w$ is reached by no generalised
cycle, is exactly the case where fullness fails, and whether the sharp value one
holds for higher-rank graphs we leave open.
\end{remark}

\begin{corollary}\label{cor:Z-stable}
Let $\Lambda$ be a row-finite $k$-graph with no sources. Suppose that:
\begin{itemize}
  \item $\Lambda$ is strongly aperiodic;
  \item every vertex in each maximal tail of $\Lambda$ is connected to by a generalised cycle; and
  \item every ascending and every descending chain of saturated hereditary subsets of $\Lambda^0$ stabilises.
\end{itemize}
Then $C^*(\Lambda)$ is $\mathcal{Z}$-stable.
\end{corollary}

\begin{proof}
By Theorem~\ref{thm:pure-inf}~(5), $C^*(\Lambda)$ is purely infinite.
Since $\Lambda$ is strongly aperiodic, every ideal of $C^*(\Lambda)$ is
gauge-invariant, so by Theorem~\ref{thm:gauge-ideals} the ideals of
$C^*(\Lambda)$ correspond bijectively to the saturated hereditary subsets of
$\Lambda^0$. The hypothesis that every ascending and every descending chain of
such subsets stabilises therefore makes $C^*(\Lambda)$ both Noetherian and
Artinian.
Since $C^*(\Lambda)$ is also separable and nuclear (see, for example, \cite{KP2,S}), 
Theorem~\ref{thm:D0-stable} applies with $D_0 = \mathcal{Z}$, and we obtain 
$C^*(\Lambda) \cong C^*(\Lambda) \otimes \mathcal{Z}$.
\end{proof}

\begin{remark}
For $k=1$ these purely infinite instances are subsumed by the recent work of
Faurot \cite{Fau}, who characterises $\mathcal Z$-stability of a graph
$C^*$-algebra through Condition (K) and a combinatorial \emph{distinct
detours} condition; his emphasis, complementary to ours, is on the stably
finite (AF) branch, where $\mathcal Z$-stability is equivalent to the absence
of elementary subquotients.
\end{remark}

\subsection*{Acknowledgement of the use of AI tools}
In preparing this article the author used an AI assistant (Anthropic's Claude)
as a tool to help brainstorm connections with an earlier manuscript and to
assist with copy-editing, notational consistency, and \LaTeX{} of the resulting
document. All mathematical content, arguments, and results are the author's own;
the author reviewed, verified, and takes full responsibility for the entire text.

\subsection*{Declaration of competing interests}
The author declares that they have no competing interests.

\end{document}